\documentclass{amsart}
\usepackage{graphicx,color}
\usepackage{mathrsfs}
\newtheorem{theorem}{Theorem}[section]
\newtheorem{pro}[theorem]{Proposition}
\newtheorem{cor}[theorem]{Corollary}
\newtheorem{lemma}[theorem]{Lemma}

\theoremstyle{definition}

\theoremstyle{remark}

\numberwithin{equation}{section}

\begin{document}

\title[ Chebyshev endomorphisms on  \({\bf \mathbb P}^3({\bf \mathbb C})\) ]{ Holomorphic endomorphisms of  \({\bf \mathbb P}^3({\bf \mathbb C})\) related to 
 a Lie algebra of type \(A_3\) and catastrophe theory}

\author{ Keisuke Uchimura}
\address{Department of Mathematics, Tokai University,  Hiratsuka, 259-1292, Japan}
\email{uchimura@tokai-u.jp. }



\begin{abstract}
The typical chaotic maps $f(x)=4x(1-x)$ and $g(z)=z^2-2$ are well known.
Veselov generalized these maps. We consider a class of maps   \(P_{A_3}^d\) of those generalized maps  and view them as holomorphic endomorphisms of  $ {\mathbb P^3}({\mathbb C})$ and make use of methods of complex dynamics in higher dimension developed by Bedford,  Fornaess, Jonsson and Sibony.  We determine    Julia sets 
$J_1, J_2, J_3, J_{\Pi}$
and the global forms of external rays.  Then we have a foliation of the Julia set  $J_2$ formed by stable disks that are composed of external rays.  

 We also show some relations between those maps and catastrophe theory. The set of the critical values of each map restricted to a real three-dimensional subspace decomposes into a tangent developable of an astroid in space and two real curves. They coincide with a cross-section of the 
set obtained by Poston and Stewart where binary quartic forms are degenerate. The tangent developable encloses the Julia set $J_3$  and
joins to a M\(\ddot{o}\)bius strip which is the Julia set  $J_{\Pi}$ in the plane at infinity in ${\mathbb P}^3({\mathbb C})$. Rulings of the 
M\(\ddot{o}\)bius strip correspond to rulings of  the surface of  $J_3$ by external rays.
\end{abstract}

\maketitle

\section{Introduction}
The typical chaotic map \enskip \(f(x) = 4x(1-x)\) \enskip is well known e.g. in \cite{UN}.  Its complex version is a Chebyshev map \enskip \(g(z) = z^2 - 2\).   It is also a chaotic map.  Generalized  Chebyshev  functions and maps in several variables were studied by several researchers, Koornwinder \cite{K},  Lidl \cite{L},    Beerends \cite{B}, Veselov \cite{V}, Hoffman and Withers \cite{HW} and Uchimura \cite{UU1}. 

A polynomial endomorphism  \(P_{A_3}^d(z_1,z_2,z_3)\) of degree \(d\) on \({\mathbb C}^3\)  is defined by the following.   We consider the \(j\)-th elementary symmetric function in \( t_1, t_2 ,t_3, t_4\)  with  \(t_4 = 1/( t_1t_2t_3)\)  for  \(j = 1, 2, 3\).
\begin{equation}
\begin{split}
\mbox{Let}\qquad \qquad \qquad\qquad \qquad z_1 = t_1 + t_2 + t_3 + \frac 1 {t_1t_2 t_3}, \\
z_2 = t_1t_2 + t_1 t_3 + t_2 t_3 + \frac 1{t_1t_2} + \frac 1{t_1 t_3} + \frac 1{t_2 t_3},\\
z_3= \frac 1{t_1} + \frac 1{t_2} + \frac 1{t_3} + t_1t_2 t_3,   \qquad (t_j \in {\bf \mathbb C}\setminus \{0\}).\\
\end{split}
\end{equation}
\[\mbox{Set}\qquad \qquad \Phi _1( t_1, t_2 ,t_3) =(z_1,z_2,z_3).\qquad\qquad \qquad\qquad \qquad\]

Then  \(P_{A_3}^d\) satisfies the  following commutative   diagram :
\def\spmapright#1{\smash{
  \mathop{\hbox to 1.3cm{\rightarrowfill}}
   \limits^{#1}}}
 \def\sbmapright#1{\smash{
  \mathop{\hbox to 1.3cm{\rightarrowfill}}
   \limits^{#1}} } 
\def\lmapdown#1{\Big\downarrow \llap{$\vcenter{\hbox{$\scriptstyle#1\,$}}$}}
\def\rmapdown#1{\Big\downarrow \rlap{$\vcenter{\hbox{$\scriptstyle#1\,$}}$}}
\def\lmapup#1{\Big\uparrow \llap{$\vcenter{\hbox{$\scriptstyle#1\,$}}$}}
\def\rmapup#1{\Big\uparrow \rlap{$\vcenter{\hbox{$\scriptstyle#1\,$}}$}}
\begin{equation}
\begin{split}
\begin {array}{ccc}
  (t_1, t_2, t_3) & \spmapright{ }& (t_1^d, t_2^d, t_3^d)\\
  \lmapdown{\Phi _{1} \enskip} &  & \lmapdown{\Phi _{1} \enskip} \\
  (z_1,z_2,z_3) & \sbmapright{} & P_{A_3}^d(z_1,z_2,z_3) \enskip.
  \end{array}
\end{split}
\end{equation}
Clearly, \(\Phi _{1}\) \enskip is a branched covering map.
We show two examples :
   \[ P_{A_3}^2(z_1,z_2,z_3) = (z_1^2-2z_2, z_2^2-2z_1z_3+2, z_3^2-2z_2),\]
 \[P_{A_3}^3(z_1,z_2,z_3) = (z_1^3-3z_1z_2+3z_3, z_2^3-3z_1z_2z_3+3z_3^2+3z_1^2-3z_2,\] \[z_3^3-3z_3z_2+3z_1).\]
These are based on the definition of Veselov \cite{V}.
Veselov \cite{V}  defined generalized Chebyshev maps as follows.  Let \(G\) be a simple complex Lie algebra of rank \(n\), \(H\) be its Cartan subalgebra, \(H^*\) be its dual space, \(\mathcal{L}\) be a lattice of weights in \(H^*\) generated by the fundamental weights \(\varpi _1, ... , \varpi_n\) and \(L\) be the dual lattice in \(H\).  One defines 
\[\phi_G : H/L \rightarrow {\bf  \mathbb C}^n, \phi_G = (\varphi_1, ... , \varphi _n), \enskip \varphi_k = \sum_{w \in W}\exp[2\pi i w(\varpi_k)],\] where \(W \) is the Weyl group, acting on the space \(H^*\).  

With each  \(G\) of rank \(n\) is associated an infinite series of integrable polynomial mappings \(P_G^d\) from \({\bf  \mathbb C}^n\) to \({\bf  \mathbb C}^n,  d = 2, 3, ... \), determined by the condition:
\[\phi _G(dx) = P_G^d(\phi _G(x)).\] 
For \(n = 1\) there is a unique simple algebra \(A_1\).  Here \(\phi _{A_1} = 2\cos(2\pi x)\) and the \(P_{A_1}^d\) are, within a linear substitution, Chebyshev polynomials of a single variable.  Here \(A_n\) is the Lie algebra of \(SL(n+1, {\bf  \mathbb C})\).

The dynamics of \(P_{A_2}^d\) is studied in \cite{U}.  In   this paper, we consider maps \(P_{A_3}^d\)  and view them as holomorphic endomorphisms of  ${\mathbb P}^3({\mathbb C})$ and make use of methods of complex   dynamics in higher dimension developed by Fornaess and Sibony \cite{FS},  and Bedford and Jonsson \cite{BJ}.  

In this paper we will provide a typical example of complex dynamics in higher dimension.  In this higher dimensional dynamics, classical geometrical figures, e. g.,  a M\(\ddot{o}\)bius strip and a special ruled surface (tangent developable) which is called  the 'Holy Grail' in catastrophe theory appear with their chaotic dynamical structures.

The main tools used in this paper are Julia sets and external rays.

We present some background on Julia sets.  The main references are  \cite{BJ},  \cite{FS} and  \cite{S}.  Let \enskip \( f :  {\bf  \mathbb C}^k \to  {\bf  \mathbb C}^k\) \enskip be a regular polynomial endomorphism of degree \(d\) (see the paragraph before Proposition 2.1).    Set 
\[K(f) : = \{z \in  {\bf  \mathbb C}^k : \{f^n(z)\} \quad \mbox{ is bounded}\}.\]
We define the Green function of \(f\) as 
\[G(z) : = \lim_{n\to \infty}d^{-n} \log^+ \Vert f^n(z)\Vert, \quad z \in  {\bf  \mathbb C}^k.\]
The Green current \enskip \(T_{{\bf  \mathbb C}^k} : = \frac 1{2\pi} dd^cG\) \enskip is a positive closed (1,1)-current.  A regular polynomial endomorphism \(f\) extends to a holomorphic endomorphism of \( {\bf  \mathbb P}^k\),  still denoted by \(f\).  The Green current  \(T_{{\bf  \mathbb C}^k}\) has an extension as a positive closed current to \({\bf  \mathbb P}^k\) in the following manner.
Every holomorphic endomorphism  \(f\) of  \({\bf \mathbb P}^k\) has a lift  \(F : {\bf  \mathbb C}^{k+1}  \to {\bf  \mathbb C}^{k+1}. \)   The projection   
 \(\pi : {\bf \mathbb C}^{k+1} \setminus \{0\} \to  {\bf  \mathbb P}^{k}\) semiconjugates \(F\) to \(f : \pi \circ F = f \circ \pi.\) 
The  Green function \(G_F\) of \(F\) is defined by
\[G_F : = \lim_{n\to \infty}d^{-n} \log \Vert F^n(z) \Vert.\]
The Green current \enskip  \(T = T_{{\bf  \mathbb P}^k}\) of \(f\) is defined by
\[\pi^*T  = \frac 1{2\pi} dd^c G_F.\]
We can define the currents \(T^l : = T \land  . . . \land T\) (\(l\) terms).  The \(l\)-th Julia set \(J_l(f)\) is the support of \(T^l\).  
The Green measure  \(\mu_f\) of \(f\) is defined by
\[\mu _f  : = (T)^k.\]
The measure \(\mu_f\) is a probability measure that is invariant under  \(f\) and maximizes entropy.

In our case we consider four kinds of Julia sets  \(J_1(f)\),   \(J_2(f)\),   \(J_3(f)\)  and   \(J_2(f_{\Pi})\),  where      \(f_{\Pi}\)  denotes the restriction of \(f\)  to the hyperplane \(\Pi\) at infinity.  We will determine those four kinds of Julia sets in Theorems 2.7,  3.2 and 4.2.

We will determine the Julia set \(J_3(f)\)  and the maximal entropy measure \(\mu_f\)  in Theorem 2.7.
The Julia set \(J_3(f)\)  coincides with the set \(K(f)\).
To obtain Theorem 2.7 we use a Briend and Duval's theorem in complex dynamics and some results of the theory of Lie groups.
 
We will determine the Julia set  \(J_2(f_{\Pi})\)   and the maximal entropy measure  \(\mu_{f_{\Pi}}\) in Theorem 3.2.  The Julia set  \(J_2(f_{\Pi})\)  is  a M\(\ddot{o}\)bius strip \(\mathcal{M}\).  On the   M\(\ddot{o}\)bius strip \(\mathcal{M}\)  we give a dynamical measure.  The map  \(f_{\Pi}\) restricted to  \({\bf \mathbb C}^2\)  is a polynomial skew product map of  \({\bf \mathbb C}^2\).  The maximal entropy measure for \(f_{\Pi}\)  restricted to the base curve which is a unit circle is \(d\theta/2\pi\) and that restricted to each ruling is the invariant measure of Chebyshev maps in one variable.  

Next we provide  some background on external rays.  External rays play an important role in the theory of dynamics in one complex variable.   Let \( f :  {\bf  \mathbb P} \to  {\bf  \mathbb P}\) \enskip be a monic polynomial map of degree \(d \ge 2\).    Suppose that the set \(K = K(f) \) is connected.  Then the complement   \({\bf \mathbb C} \setminus K\)  is conformally isomorphic to the complement   \({\bf \mathbb C} \setminus {\bar{\bf \mathbb D}}\)   under the  B\(\ddot{o}\)ttcher map \(\phi\).  The external rays for \(K\) are defined by 
\[\{z : \arg (\phi(z)) = constant\}.\]
The image of an external ray under \(f\) is also another  external ray.

Bedford and Jonsson \cite{BJ} define external rays for holomorphic endomorphisms of  \({\bf  \mathbb P}^k\). 
We will determine the global forms of external rays of our maps \(f = P_{A_3}^d\).  The image of each external ray under the extended map \(f\) on \({\bf  \mathbb P}^3\)  is also an external ray.   We will show in Theorem 4.2 that the Julia set \(J_2(f)\)  is a foliated space and leaves of the space are stable disks composed of external rays.   The image of a stable disk  under the map \(f\)  is another stable disk.

Next we consider the dynamics of \(P_{A_3}^d\) restricted to a real three-dimensional subspace.
The map \enskip \(P_{A_3}^d :  {\mathbb C}^3 \to {\mathbb C}^3\) \enskip admits an invariant space 
\[R_3 : = \{(z_1, z_2, z_3) \in {\mathbb C}^3  : z_1 = {\bar z}_3 \enskip \mbox{and} \enskip z_2  \enskip \mbox{is real}\}.\]
We consider the dynamics of \(P_{A_3}^d\) restricted to \(R_3\).  The set \enskip \(J_3(f) = K(f)\)\enskip lies in the space \(R_3\).  Sometimes we may regard  \(R_3\) as  \({\bf  \mathbb R}^3\).  Then  \(J_3(f)\) is isomorphic to a closed domain  in \({\bf  \mathbb R}^3\)  bounded by the ruled surface \(\mathcal{A}\) whose base curve is an astroid in space (see Proposition 2.4 and Figure 3). 
In particular,   \(\mathcal{A}\) is a part of the tangent developable of an astroid in space and so we call it an astroidalhedron.   A ruled surface is called a tangent developable if its rulings are tangent lines to its base curve.   The  ruled surface \(\mathcal{A}\) has a relationship to the root system of Lie algebra of type  \(A_3\)  and a \((\sqrt 3, \sqrt 3, 2)-\)tetrahedron (see Figure 2). 

The external rays included  in \(R_3\)  are half-lines that connect the  ruled surface \(\mathcal{A}\)  and the M\(\ddot{o}\)bius strip \enskip \(\mathcal{M} = J_2(f_{\Pi})\).  By this fact,  we will show  that rulings of \(\mathcal{M}\) correspond to rulings of \(\mathcal{A}\)  by external rays in Proposition 4.9.

Next we will show some relations between those maps and catastrophe theory.  

The dynamics of the  maps  \( P^d_{A_2} \) on \( {\mathbb C}^2\)  is studied  in \cite{U}. 
 The set of critical values of \( P^d_{A_2} \) restricted to 
 \(\{z_1 = {\bar z}_2\}\) is proved to be  a deltoid.  The deltoid coincides with a cross-section of the bifurcation set (caustics) of the elliptic umbilic catastrophe map  \((D_4^-)\).   In \cite{U}, it is shown that  the external rays and their extensions  constitute a family  of lines whose envelope is the deltoid.  Hence these lines are real 'rays' of caustics.   See Figure 9.

In addition to the caustics, the deltoid has relations with binary cubic forms\\
\qquad \(f(x,y) = ax^3 + bx^2y + cxy^2 + dy^3, \quad a,b,c,d \in {\mathbb R}.\)\\
Let V be the set where the discriminant of \(f(x,y)\) vanishes.  To understand the geometry of the set V, Zeeman\cite{Z} pursues a different tack.  
Zeeman\cite{Z} shows that \(V \cap S^3\)  is mapped diffeomorphically to the 'umbilic bracelet'.  It has a deltoid section that rotates \(1/3\) twist going once round the bracelet.

We return to the study of the maps \(P_{A_3}^d\).  In this case we will show that the set of critical values of \(P_{A_3}^d\) restricted to \(R_3\) has relations with binary quartic forms.     

Poston and Stewart  study  quartic forms in two variables in \cite{PS1} and \cite{PS2}\\
\qquad \(f(x,y) = ax^4 + 4bx^3y + 6cx^2y^2 + 4dxy^3 + ey^4, \quad a,b,c,d,e \in {\mathbb R}. \)\\
Let \(\triangle\) be the discriminant of \(f(x,y)\) and  \(\mathscr{D} \subset {\mathbb R}^5\)     be the algebraic set given by \(\triangle = 0\).  The set \(\mathscr{W} = \mathscr{D} \cap S^5\)  is decomposed  into \(\mathscr{W}_1\) and \(\mathscr{W}_{\infty}\).  \(\mathscr{W}_1\) is diffeomorphic to  \(\mathscr{U}\).
They consider  a cross-section \(\mathscr{Q}\) of \(\mathscr{U}\).  The shape for \(\mathscr{Q}\) is called the 'Holy Grail' in catastrophe theory.  We will show in Proposition 5.8 that the set \(\mathscr{Q}\) coincides with the set of critical values of \(P_{A_3}^d\) restricted to \(R_3\)  by a coordinate transformation.  We will show that the set  decomposes into a tangent developable \(\mathcal{T}\) of an astroid in space and two real curves in Proposition 5.5.  See Figure 10.  The astroidalhedron \(\mathcal{A}\) is a part of \(\mathcal{T}\).

In Proposition 5.6, we will show that the rims of  \(\mathcal{T}\) join simply to the boundary of  \(\mathcal{M}\) in the hyperplane \(\Pi\) at infinity in \( {\mathbb P}^3( {\mathbb C})\).  Poston and Stewart  deal with the same situation by analyzing \(\mathscr{W}_{\infty}\) in  \({\mathbb R}^5\)  in \cite{PS1} and \cite{PS2}.  It is complicated.  But we consider the situation in \( {\mathbb P}^3( {\mathbb C})\) and so our description is simpler.
We will show that any ruling of \(\mathcal{T}\)  i.e. any tangent line to the astroid consists of two external rays and their extension and that any external ray which is not a ruling connects the astroidalhedron \(\mathcal{A}\) and M\(\ddot{o}\)bius strip  \(\mathcal{M}\).  

 In this paper,  we will  show not only static aspects of catastrophe theory but also dynamical aspects of catastrophe theory.  We know that the sets of critical values of \(P_{A_2}^d \) and \(P_{A_3}^d\) restricted to the real subspaces have relations with binary cubic forms and quartic forms,  respectively.   These relations will be generalized for general  maps  \(P_{A_n}^d\).

\section{The sets \(K(P_{A_3}^d) \) and \(J_3(P_{A_3}^d)\)}

In this section we determine the set  \(K(P_{A_3}^d) \)  of   bounded orbits and the third Julia set \(J_3(P_{A_3}^d)\).  We will show that the surface of \(K(P_{A_3}^d) \)  is a part of the tangent developable of an astroid in space. 

We consider the map  \(P_{A_3}^d\) defined by (1.1) and (1.2).

\[\mbox{Let} \quad P_{A_3}^d  = (g_1^{(d)}(z_1,z_2,z_3), g_2^{(d)}(z_1,z_2,z_3), g_3^{(d)}(z_1,z_2,z_3)).\]
Then, from \cite{L}(pp. 183-184)  we know that the set of polynomials \(\{g_j^{(d)}(z_1,z_2,z_3)\}\) satisfies the following  recurrence formulas :
\begin{equation}
\begin{split}
g_1^{(k)} = z_1g_1^{(k-1)} - z_2g_1^{(k-2)} + z_3g_1^{(k-3)} - g_1^{(k-4)} ,\\
g_1^{(j)} = \sum_{r=1}^{j} (- 1) ^{r-1}z_rg_1^{(j-r)} + (-1)^j(4-j) z_j, \quad (j = 0, 1, 2, 3), \quad z_0 = 1.
\end{split}
\end{equation}
\begin{equation}
 g_3^{(k)}(z_1,z_2,z_3) =  g_1^{(k)}(z_3,z_2,z_1).
\end{equation}
\begin{equation}
\begin{split}
g_2^{(k+6)} - z_2g_2^{(k+5)} + (z_1z_3-1)g_2^{(k+4)} - (z_1^2-2z_2 + z_3^2)g_2^{(k+3)}\\
 + (z_1z_3-1)g_2^{(k+2)} -  z_2g_2^{(k+1)} + g_2^{(k)} = 0 .
\end{split}
\end{equation}
Note that the formula in[14, p. 184] corresponding to (2.3) is incorrect.  The correct coefficient of \(g_2^{(k+3)}\)  is equal to \(-(z_1^2 - 2z_2 + z_3^2)\).\\
And the correct initial values are given by
 \[g_{2}^{(-2)} = z_2^2-2z_1z_2+2,\quad g_{2}^{(-1)} = z_2, \quad g_{2}^{(0)} = 6, \quad g_{2}^{(1)} = z_2,\]
\[g_{2}^{(2)} = g_{2}^{(-2)}, \quad g_{2}^{(3)} =z_2^3-3z_1z_2z_3+3z_3^2+3z_1^2-3z_2.\]

A polynomial endomorphism \(f\) of degree \(d\) is called {\it regular}  if the homogeneous part \(f_h\) of degree \(d\) satisfies \(f_h^{-1}(0) = \{0\}.\) 

 \begin{pro} \label{pro:A}
  \(P_{A_3}^d(z_1,z_2,z_3)\) is a regular polynomial endomorphism.
\end{pro}
\begin{proof}
Let   \(f : = P_{A_3}^d(z_1,z_2,z_3).\)   From (2.1), (2.2) and (2.3),  we have   \(f_h = (z_1^{d}, h_2^{(d)},z_3^{d})\),  where \(h_2^{(d)}(z_1,z_2,z_3)\)  is a polynomial satisfying the recurrence formula :
\begin{equation}
\begin{split}
h_2^{(d+2)} = z_2h_2^{(d+1)} - z_1z_3h_2^{(d)}, \\
 h_2^{(1)} = z_2, \quad h_2^{(2)} = z_2^2 -2z_1z_3.
\end{split}
\end{equation}

Then we deduce \quad \(f_h^{-1}(0) = \{0\}\).
\end{proof}

Next we study the set 
 \[K( P_{A_3}^d) = \{  z \in {\bf \mathbb C}^3 : \mbox{the orbit} \quad \{( P_{A_3}^d)^n(z)\}\quad \mbox{is   bounded} \}.\]
Then \(K( P_{A_3}^d)\) \quad is described in the following form.

 \begin{pro} \label{pro:B}(\cite{V})
  \(K( P_{A_3}^d) = \{\Phi _{1}(t_1, t_2, t_3) : \mid t_1\mid = \mid t_2\mid = \mid t_3\mid =1 \}\). 
\end{pro}

The set   \(K(P_{A_3}^d(z_1,z_2,z_3)) \) is given by
\begin{equation}
\begin{split}
\left \{
   \begin{array}{lll}
   z_1 = e^{i\alpha } + e^{i\beta } + e^{i\gamma } + e^{i( -\alpha -\beta -\gamma) } ,\\
   z_2 = e^{i(\alpha +\beta )} +e^{i(\alpha +\gamma  )} + e^{i(\gamma +\beta )} +e^{-i(\beta+\gamma  )}+ e^{-i(\gamma +\alpha  )} + e^{-i(\alpha +\beta )} ,\\
   z_3 = e^{-i\alpha } + e^{-i\beta } + e^{-i\gamma } + e^{i(\alpha +\beta +\gamma )} ,
   \end{array}
   \right.
\end{split}
\end{equation}

   \[-\alpha -\beta -\gamma  \leq \alpha   \leq \beta \leq \gamma   \leq 2\pi  -\alpha -\beta  - \gamma .  \quad \mbox{See \cite{EL}}.\]
We call
\(R' := \{(\alpha , \beta , \gamma ) : -\alpha -\beta -\gamma \leq  \alpha   \leq \beta \leq  \gamma   \leq 2\pi  -\alpha -\beta  - \gamma \}   \)
the {\it natural domain}.\\
\begin{figure}[htbp]
\begin{tabular}{cc}
\begin{minipage}{0.5\hsize}
\begin{center}
\includegraphics[scale=0.29]{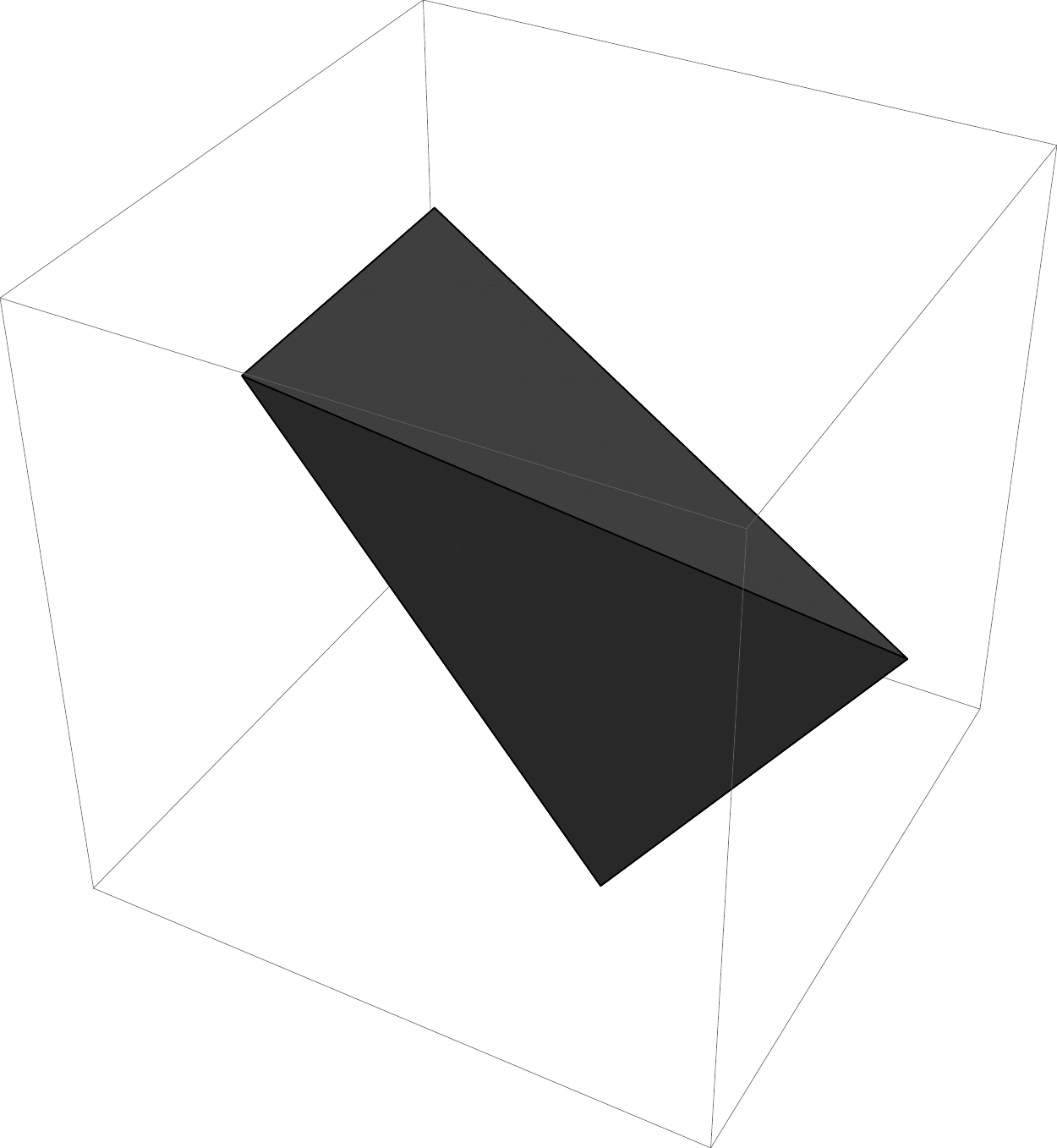} \\
\caption{ The natural domain \(R'\).  }
\label{fig1}
\end{center}
\end{minipage}
\begin{minipage}{0.55\hsize}
\begin{center}
\includegraphics[scale=0.29]{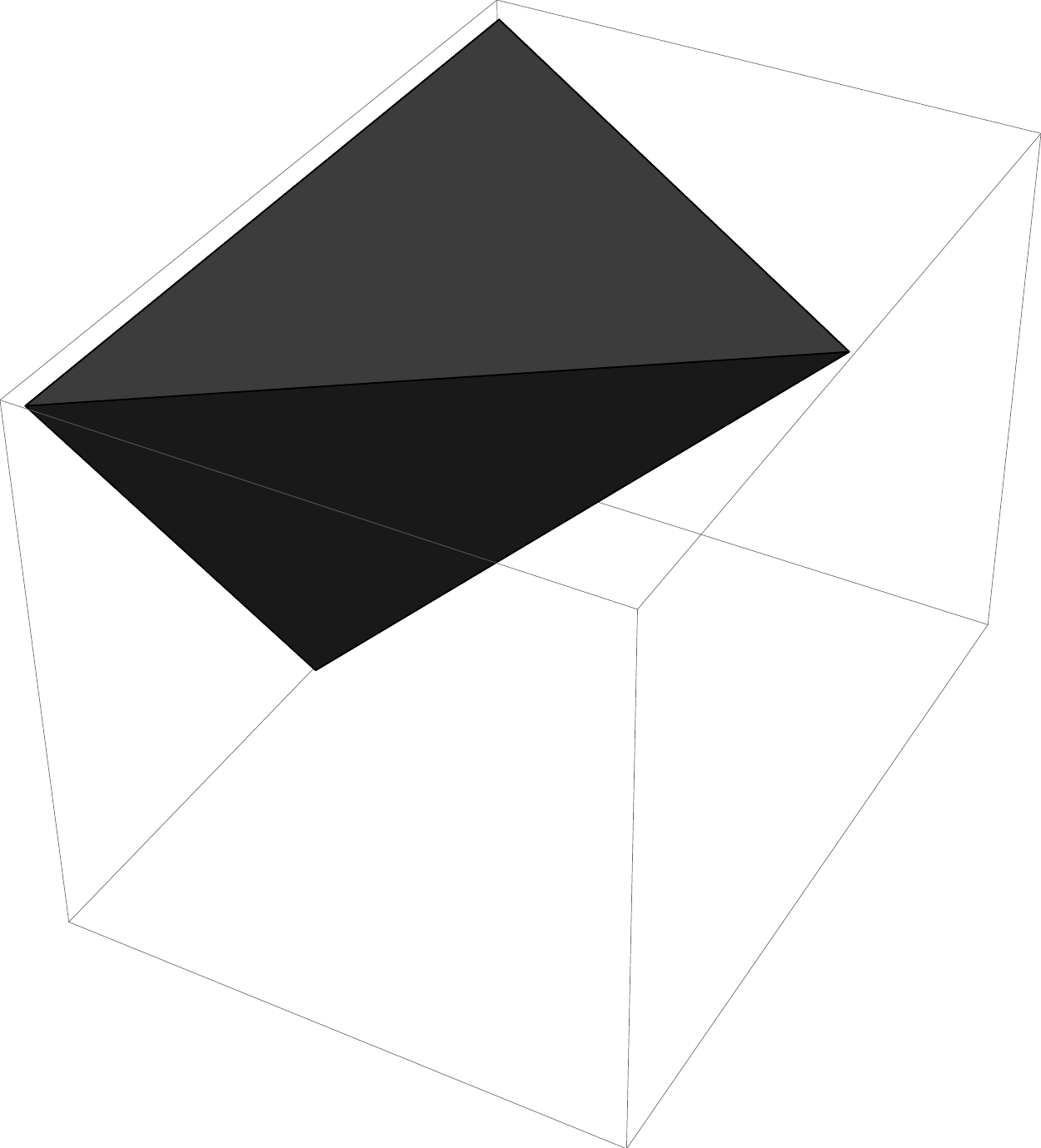}\\
\caption{ The fundamental region $R$.}
\label{kfig2..eps}
\end{center}
\end{minipage} 
\end{tabular}
\end{figure} 

We denote the real three-dimensional subspace
  \(\{(z_1,z_2,\bar z_1) : z_1 \in {\bf \mathbb C},  z_2 \in {\bf \mathbb R}\}\) by \(R_3\).  
Then
 \( K( P_{A_3}^d)  \subset  R_3\) . 
\(R_3\) is invariant under the maps  \(P_{A_3}^d\) .  Sometimes we regard \(R_3\) as \({\bf \mathbb R}^3\) .

In order to facilitate computations we transform the Euclidean coordinates \((\alpha , \beta , \gamma )\) into new coordinates \((s_1, s_2, s_3)\) concerning the root system of type \(A_3\). 
 
A base \(\{\alpha_j\}\) for the root system and fundamental weights \( \varpi  _j\)  of type \(A_3\) are given by
 \[\alpha _1 =  (-\frac{1}{\sqrt 2},  -1, \frac 1{\sqrt 2}), \enskip \alpha _2 = (\sqrt 2, 0, 0),\enskip \alpha _3 = (-\frac{1}{\sqrt 2},  1, \frac 1{\sqrt 2}), \]
 \[\varpi  _1 = (0,  -\frac 1{2}, \frac 1{\sqrt 2} ), \enskip \varpi  _2 = (\frac{1}{\sqrt 2},  0, \frac 1{\sqrt 2}), \enskip \varpi  _3 = (0, \frac{1}{ 2},  \frac 1{\sqrt 2} ).\]

One of the alcoves of \(A_3\) is the closed region \(R\) bounded by the polyhedron \(\sqrt 2\pi\)  \((O, \varpi _1, \varpi _2, \varpi _3)\). We call the region \(R\) the {\it fundamental region}.  The region  \(R'\) is transformed to  \(R\) by a transformation \(T\).
The matrix associated with the transformation \(T\) from the \((\alpha, \beta, \gamma)\)  space to \((s_1, s_2, s_3)\) space is given by
\begin{equation}
\begin{split}
\left(\begin{array}{cc} s_1\\ \\ s_2\\ \\ s_3\\  \end{array}
\right) =
\left(\begin{array}{ccc} -\frac 12, & \frac 12, & 0
\\ \\ -\frac 1{\sqrt2}, & -\frac 1{\sqrt2}, & 0
\\ \\ \frac 12, & \frac 12, & 1\\  \end{array} \right)
\left(\begin{array}{cc} \alpha\\ \\ \beta\\ \\ \gamma\\  \end{array}
\right). 
\end{split}
 \end{equation}

The region  \(R\) is a closed region bounded by a \((\sqrt 3,  \sqrt 3, 2)\)-{\it tetrahedron}.  That is, it has four faces which are congruent with each other and the ratios of whose edge lengths are equal to \(\sqrt 3 :\sqrt 3 : 2\).  Coxeter\cite{Co}  proved that there exist only seven types of reflective space-fillers.  It is one of them.  A convex polyhedron \(P\)  is called a reflective space-filler if its congruent copies tile the 3-space in such a way that\\
\quad (1) the tilling is face-to face,\\
\quad (2) if the intersection  \(P_1 \cap P_2\)  of two of those copies has a face in common, \\
\qquad then   \(P_1 \) is the mirror-image of  \(P_2\)  in the common face, and\\
\quad (3) each of the dihedral angels of  \(P\) is  \(\pi/k\)  for integer  \(k \ge 2\).

We consider the tilling of the \((s_1, s_2, s_3)\) space by   
\((\sqrt 3, \sqrt 3, 2)\)-tetrahedrons.  The region \(R\) is a closed region bounded by one of these tetrahedrons with vertices 
\[O = (0, 0, 0), \quad A_1 = (0, -\pi/\sqrt 2, \pi), \quad A_2 = (\pi, 0, \pi), \quad A_3 = (0, \pi/\sqrt 2, \pi).\]
 
Let  \(\mathcal{G}\) be the group of isometrics which is generated by the reflections in the faces of these tetrahedrons. 

The reflection in the hyperplane through the origin orthogonal to \(\alpha_i\) is given by 
\[w_{\alpha_i}(x) = x - \frac{2(x, \alpha_i)}{(\alpha_i, \alpha_i)}\alpha_i, \quad (i = 1, 2, 3), \quad x \in   {\mathbb R^3}.\]
Set \enskip \(J_i : = w_{\alpha_i}\).  Then  \(J_i\) is the reflection in the face \(\triangle OA_jA_k\) of the tetrahedron \(\partial R\) with \(\{i,  j, k\} = \{1, 2, 3\}\).  Set   \(J_0(s_1, s_2, s_3) = (s_1, s_2, 2\pi - s_3)\).  Then \(J_0\) is the reflection in the face \(\triangle A_1A_2A_3\).  It is known e.g. in  \cite{Bo}  that the reflections \(J_0, J_1, J_2\)  and  \(J_3\)  generate the group \(\mathcal{G}\).  Set \enskip \(X = \{e^{i\alpha},  e^{i\beta},  e^{i\gamma}, e^{-i(\alpha+\beta+\gamma)}\}\).
Then by the direct computations using (2.6) we can prove that each \(J_k\) acts on the set \(X\) as a permutation, for \(k = 0, 1, 2, 3.\)    For any element  \((s_1, s_2, s_3)\) in the  space,  these exists an element \(J\) in the group \(\mathcal{G}\) such that \enskip \(J(s_1, s_2, s_3) \in R\). 
 \begin{pro} \label{pro:B2}
 For \(k = 0, 1, 2, 3,\)  let the images of  \((s_1, s_2, s_3)\) and 
 \(J_k(s_1, s_2, s_3) \) under the  inverse of the transformation  \(T\) be  \((\alpha, \beta, \gamma)\) and \((\alpha', \beta', \gamma')\).  Then we have \enskip \[\Phi_1(e^{i\alpha},  e^{i\beta},  e^{i\gamma}) = \Phi_1(e^{i\alpha'},  e^{i\beta'},  e^{i\gamma'}).\]
\end{pro}
\begin{proof}
The terms in  \(z_i\enskip (i = 1, 2, 3)\) in (2.5) are invariant under any \(J_k\).
\end{proof}
 We study the surface of \(K(P_{A_3}^d)\).
We define a coordinate system \((p_1, p_2, q)\) of \(R_3\) by
\[p_1(1,0,0,0,1,0) + p_2(0,1,0,0,0,-1) + q(0,0,1,0,0,0).\]
We consider the map \(\Phi _{1} \enskip\)  restricted to \(R'\) onto \enskip \(K(f) \subset R_3\).   We denote it by  \(\varphi_1\).  The mapping \quad \(\varphi_1 : R' \to K(f)\) \quad is given by
\begin{equation}
\begin{split}  
   p_1 = Re(e^{i\alpha } + e^{i\beta } + e^{i\gamma } + e^{i( -\alpha -\beta -\gamma) }) ,\qquad \qquad \qquad \qquad\\
 p_2 = Im(e^{i\alpha } + e^{i\beta } + e^{i\gamma } + e^{i( -\alpha -\beta -\gamma) }) , \qquad \qquad \qquad \qquad\\
   q = e^{i(\alpha +\beta )} +e^{i(\alpha +\gamma  )} + e^{i(\gamma +\beta )} +e^{-i(\beta+\gamma  )}+ e^{-i(\gamma +\alpha  )} + e^{-i(\alpha +\beta )} .\\
\end{split}   
\end{equation}
\(\varphi_1\)  is a diffeomorphism from \(int(R')\) to \(int(K(f))\)  and \(\partial R'\) is mapped onto \(\partial K(f)\) injectively.
\begin{pro} \label{pro:C}.
The surface of \(K(P_{A_3}^d)\) is a part of the tangent developable of an astroid in space.   The surface is given by 
\[\chi(u,v) = (4\cos^3 u , 4\sin^3 u,  6\cos 2u) + v(\cos u ,  -\sin u , 2 ) ,\qquad \qquad \]
\[ (-2-2\cos 2u \leq v  \leq  2-2\cos 2u ).\qquad \qquad \qquad \qquad\]
\end{pro}
\begin{proof}
To get the surface,  we substitute an inequality sign for an equality sign in the definition of \(R'\).  That is, we set \quad \(-\alpha-\beta-\gamma = \alpha\).   By (2.7) and the above equality, we have

\begin{equation}
\begin{split}
(p_1, p_2, q) = 2(\cos\alpha, \sin\alpha, \cos2\alpha) + 2\cos(\alpha+\beta)(\cos \alpha, -\sin\alpha, 2), \\ 
(0 \le \alpha < 2\pi, \quad 0 \le \alpha+\beta < \pi).
\end{split}
\end{equation}
From the properties of reflections of \(R\),  we see that (2.8) represents the surface of \(K(P_{A_3}^d)\).    
It is a ruled surface.  Using a striction curve (\cite{G}, 17.3),  
we reparametrize the ruled surface.
Set
\[\tilde{\chi}(u,v) = 2(\cos u , \sin u,  \cos 2u) + 2v(\cos u ,  -\sin u , 2 ) .\qquad \qquad \]
Then from Lemma 17.7 in \cite{G}, we have a reparametrization
\[\chi(u,v) = (4\cos^3 u , 4\sin^3 u,  6\cos 2u) + v(\cos u ,  -\sin u , 2 ) ,\qquad \qquad \]
\[(-2-2\cos 2u \le v \le 2-2 \cos 2u) .\qquad \qquad \]
The base curve \(\{(4\cos^3 u , 4\sin^3 u,  6\cos 2u) : 0 \le u < 2\pi\}\) is an astroid in space and \(\chi(u,v)\) is a part of the tangent developable of the astroid. 
\end{proof}
\begin{figure}[htbp]
\begin{tabular}{cc}
\begin{minipage}{0.45\hsize}
\begin{center}
\includegraphics[scale=0.39]{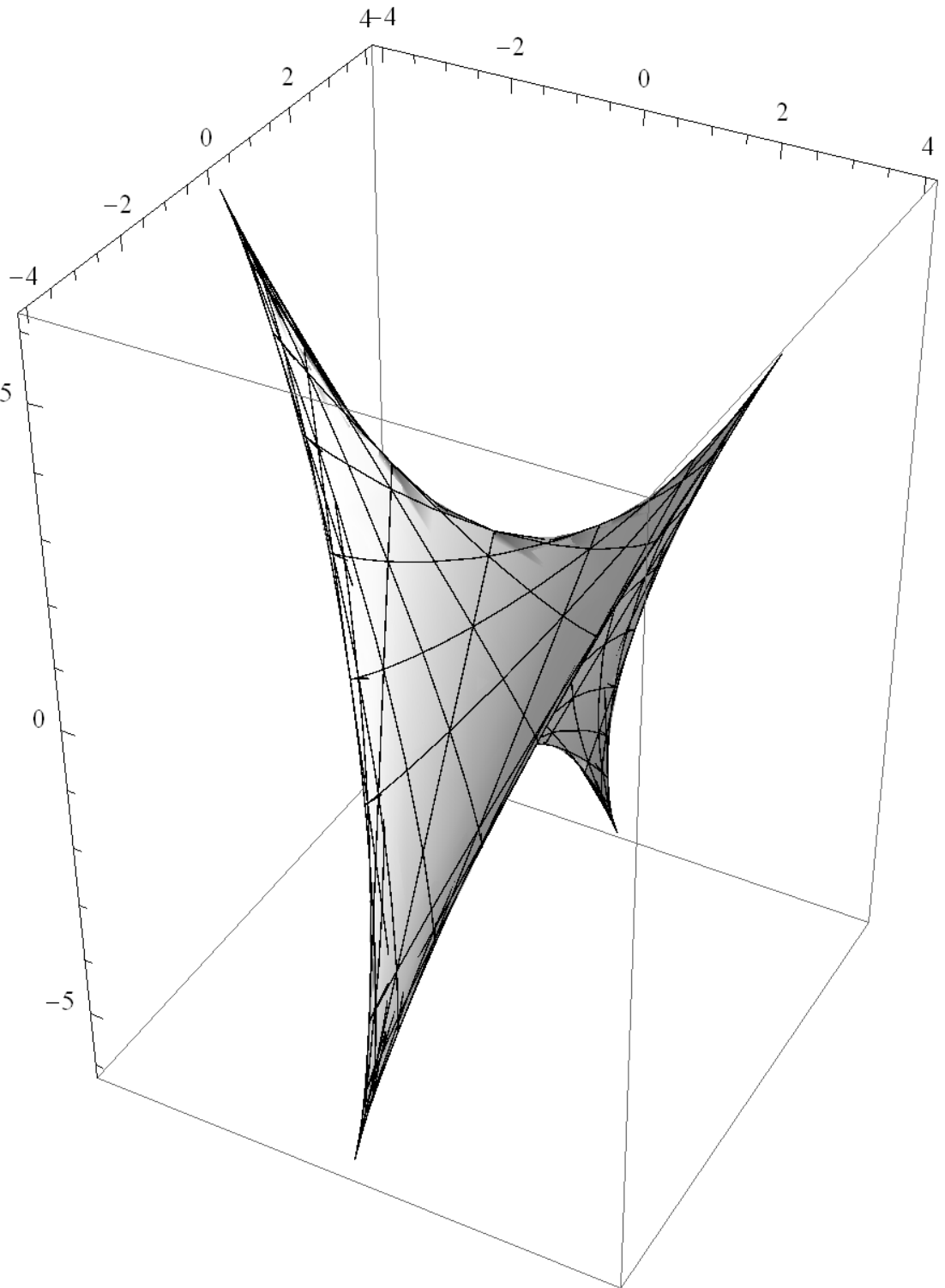} \\
\caption{ An astroidalhedron. }
\label{fig3}
\end{center}
\end{minipage}
\begin{minipage}{0.5\hsize}
\begin{center}
\includegraphics[scale=0.28]{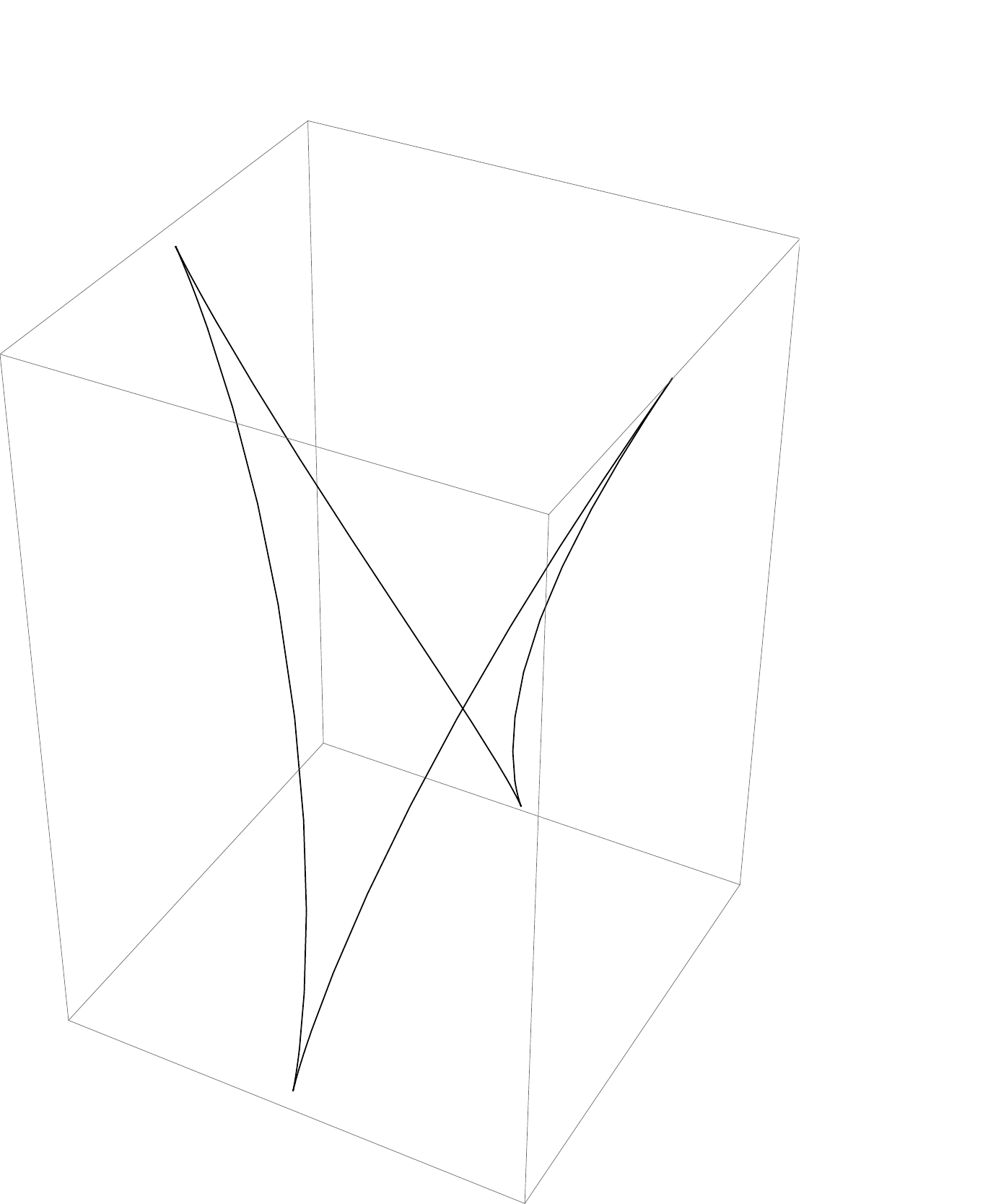}\hspace{1.5cm} \\
\caption{ An astroid in space. }
\label{fig4}
\end{center}
\end{minipage} 
\end{tabular}
\end{figure} 
The astroid consists of edges of the surface.  We call the ruled surface an {\it astroidalhedron}  and denote it by 
\(\mathcal{A}\).
By  \cite{KR}, we see that those  edges except for four  vertices of \(\mathcal{A}\) are cuspidal edges.

Now we begin with the study of Julia sets.  In Section 1 we define the \(l\)-th Julia set \(J_l\).  In our situation we have three kinds of Julia sets \(J_1, J_2\)  and  \(J_3\).   Clearly, \(J_1 \supset J_2 \supset J_3\).   We begin with the study of   \(J_3\).  We will show that  \( J_3 = K(P_{A_3}^d).\)  To show this we use a theorem of Briend and Duval \cite{BD}.  It reads as follows.
Let \(P_n\) denote the set of repelling periodic points of period \(n\).  The number of the elements in \(P_n\)  is  \(d^{3n}\).  Let  \( f = P_{A_3}^d.\)   Set  \(\mu = (T_f)^3\).
\[\mbox{Then the sequence of measures}\quad \mu_n : = d^{-3n}
\sum_{ a \in P_n} \delta_a \quad \mbox{converges weakly to}\quad \mu. \]

From the above diagram (1.2),  we have the following lemma.
\begin{lemma} \label{lemma:D}
Any periodic point of \(f\) in \(int(K(f))\) is repelling.
\end{lemma}

Next we consider the distribution of repelling periodic points.  Using a conjugacy from \(K(f)\) 
 to \(R\), we study the distribution of repelling periodic points.    We will show that the repelling periodic points are dense and equidistributed in \(R\).

  Combining the inverse of \(\varphi_1\) with the coordinate transformation \(T\),  we get a continuous map \(\varphi\)  from \(K(f)\) to \(R\) such that \(\varphi\)  restricted to \(int(K(f))\) is a diffeomorphism.  We set   \(\rho : = \varphi \circ f \circ
\varphi^{-1}\).  Then \(\rho(s_1, s_2, s_3) = d(s_1, s_2, s_3).\) 

To study the distribution of periodic points of \(\rho\),  we use an argument similar to that used in Proposition 2.2 of \cite{U}.  

We first consider the case \(d = 2\).
The image of the fundamental region \(R\) under \(\rho\) and division of it into eight  \((\sqrt 3, \sqrt3, 2)\)-tetrahedrons are depicted in Figure 5.

For any \(d \ge 3\),  we combine the three adjacent \((\sqrt 3, \sqrt3, 2)\)-tetrahedrons  which yield a triangular prism. A small ball denotes the origin.  See Figure 6.
\begin{figure}[htbp]
\begin{tabular}{cc}
\begin{minipage}{0.5\hsize}
\begin{center}
\includegraphics[scale=0.37]{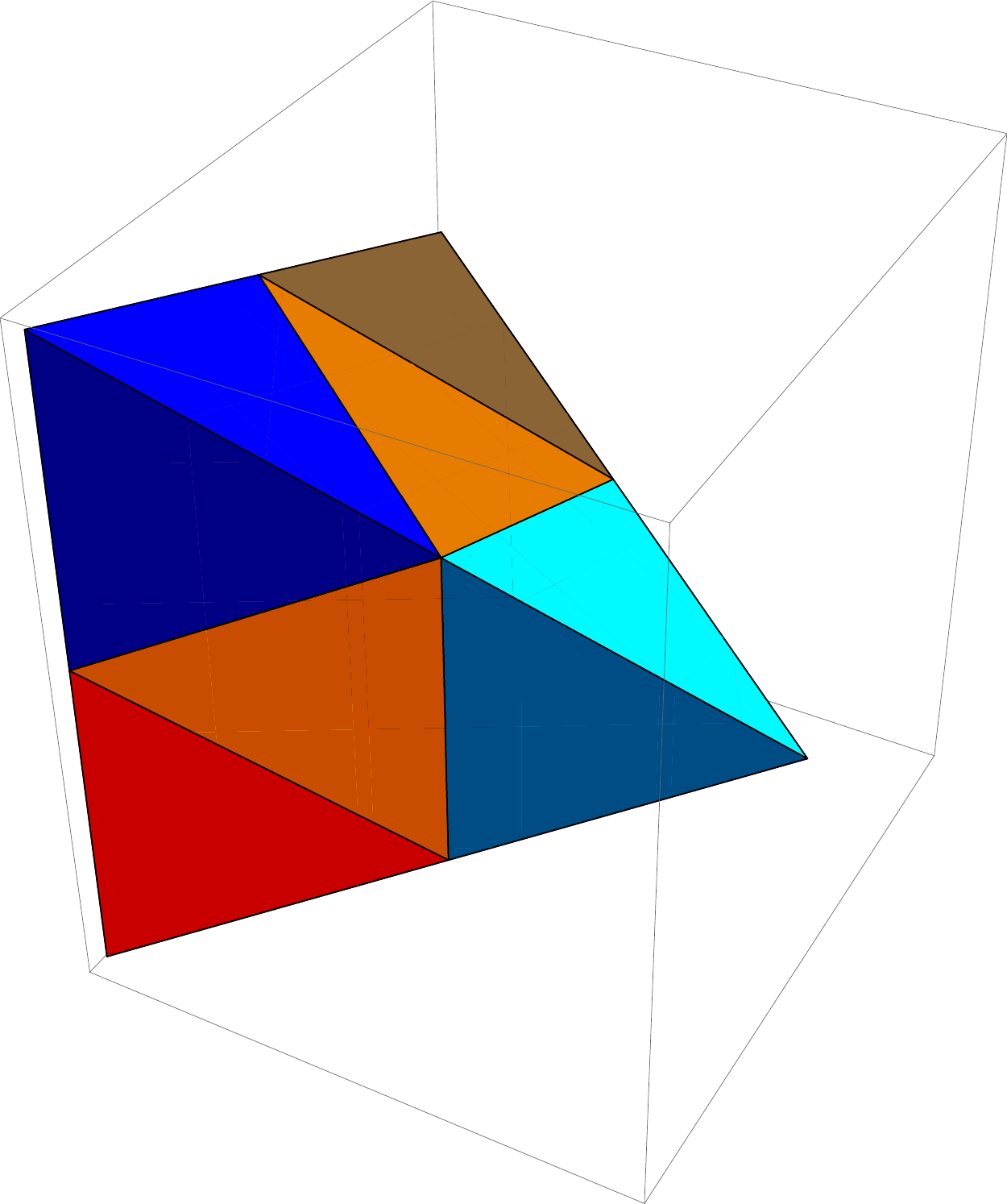} \\
\caption{ Eight tetrahedrons. }
\label{Fig1-5p}
\end{center}
\end{minipage}
\begin{minipage}{0.5\hsize}
\begin{center}
\includegraphics[scale=0.45]{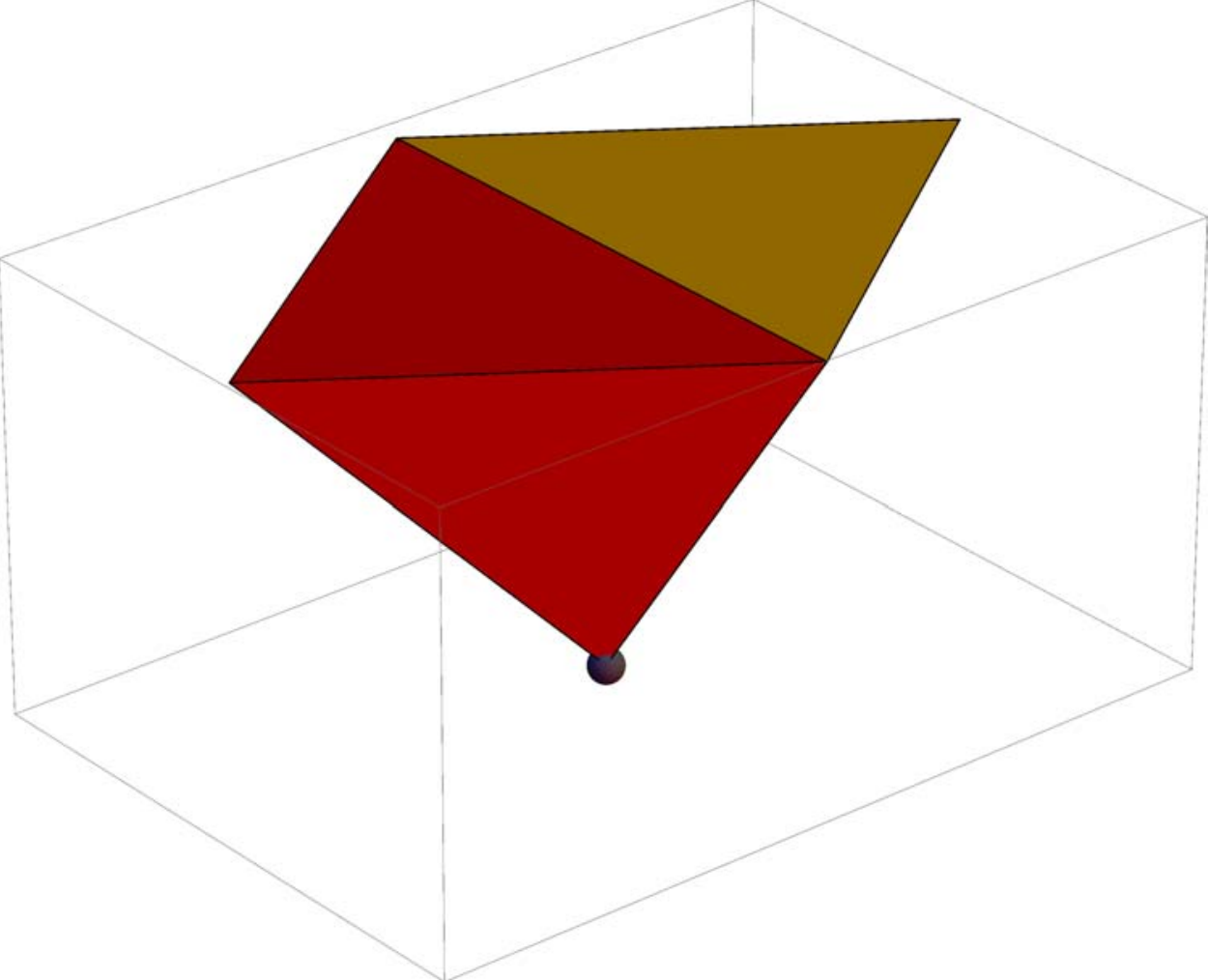}\hspace{1.5cm} \\
\caption{ A triangular prism. }
\label{Fig1-6p}
\end{center}
\end{minipage} 
\end{tabular}
\end{figure}
 
  The triangular prism plays the same role as the equilateral triangle plays in Proposition 2.2 in \cite{U}.   
Then the image of the fundamental region \(R\) under  \(\rho^n\) consists of  \(d^{3n}\) regions each of which is  congruent to \(R\).  Each region is mapped to \(R\) by some sequence of reflections in \(\mathcal{G}\).  

Conversely we consider the subdrivision  of \(R\).
We can divide the fundamental region \(R\) into \(d^{3n}\)  regions  each \(D_n\)  of which is congruent to a  region bounded by a smaller \((\sqrt 3, \sqrt3, 2)\)-tetrahedron.  Combining \(\rho^n\) and the sequence of reflections we have a continuous map from  \(D_n\) onto \(R\).  Then by the fixed point theorem,  we can prove the following lemma.
 \begin{lemma} \label{lemma:E}
Each region \(D_n\) has a periodic point  of period \(n\) of \(\rho\). 
\end{lemma}

All the repelling periodic points are dense and equidistributed in  \(R\).  Hence  we can prove the following theorem.
\begin{theorem} \label{theorem:F}
(1) \( J_3(P_{A_3}^d) = K(P_{A_3}^d).\)\\
(2) The maximal entropy measure \(\mu\) of \(P_{A_3}^d(z_1,z_2,z_3)\)  is given by 
\[\mu = \frac 3{\pi^3} \frac 1 {\sqrt{d_3}}dp_1 dp_2 dq,\]　
\[\mbox{where} \quad d_3 = 256 -27(z_1^4+{\bar z}_1^4) + (z_1^2+{\bar z}_1^2)(144z_2-4z_2^3 + 18z_1{\bar z}_1z_2)\]
\[-80z_1{\bar z}_1z_2^2+z_1^2{\bar z}_1^2z_2^2-192z_1{\bar z}_1-4z_1^3{\bar z}_1^3-6z_1^2{\bar z}_1^2 -128z_2^2 + 16z_2^4,\]

with \(z_1 = p_1 + ip_2\)  and  \(z_2 = q.\)\\
(3)  The Lyapunov exponents of \(P_{A_3}^d\) with respect to the measure \(\mu\) are given by \(\lambda _1 = \lambda _2 = \lambda _3 = \log d\) .　\\
\end{theorem}
\begin{proof}\quad (1):
From the Briend and Duval's theorem,  Lemmas \ref{lemma:D} and  \ref{lemma:E} , we have \(J_3(P_{A_3}^d) = K(P_{A_3}^d).\)

(2): By pulling back the Lebesgue measure on \(R\)  we will obtain the invariant measure  \(\mu\).  Set \(\tilde{\mu}_n : = \varphi_*\mu_n.\)  From Lemma \ref{lemma:E}  we deduce that the sequence \(\{\tilde{\mu}_n\}\)  converges weakly to \(\tilde{\mu} =  \frac {3\sqrt 2}{\pi^3}  ds_1\wedge ds_2\wedge ds_3.\)
\[\mbox{Hence} \quad \mu = \frac {3\sqrt 2}{\pi^3} \varphi^* ds_1\wedge ds_2\wedge ds_3.\]
\[\mbox{From (2.6), we have} \quad T^* ds_1\wedge ds_2\wedge ds_3 = \frac 1{\sqrt 2}  d \alpha \wedge d\beta \wedge d \gamma . \qquad \qquad \qquad\]
\[\mbox{Using Lemma 3 in \cite{EL},  we can compute Jacobian determinant } \quad
\det\frac{\partial(p_1, p_2, q)}{\partial(\alpha, \beta, \gamma)} .\]
\[\mbox{Then } \quad
\left(\det\frac{\partial(p_1, p_2, q)}{\partial(\alpha, \beta, \gamma)}\right)^2 =  d_3 ,\qquad \qquad\]
\[\mbox{where} \quad d_3 = 256 -27(z_1^4+{\bar z}_1^4) + (z_1^2+{\bar z}_1^2)(144z_2-4z_2^3 + 18z_1{\bar z}_1z_2)\]
\[-80z_1{\bar z}_1z_2^2+z_1^2{\bar z}_1^2z_2^2-192z_1{\bar z}_1-4z_1^3{\bar z}_1^3-6z_1^2{\bar z}_1^2 -128z_2^2 + 16z_2^4,\]

with \(z_1 = p_1 + ip_2\)  and  \(z_2 = q.\)\\
(Note that the formula in p.98 of \cite{EL}  corresponding to the above formula for \(d_3\)  is incorrect.)
\[\mbox{Hence} \quad  (\varphi_1^{-1})^* d \alpha \wedge d\beta \wedge d \gamma = \frac 1{\sqrt{ d_3}}dp_1\wedge dp_2\wedge dq   .\]
Since  \(\varphi^{*} =  (\varphi_1^{-1})^*T^*\),  the assertion (2) follows.  

(3): The assertion (3) follows from the fact that 
 \(\rho(s_1, s_2, s_3) = d(s_1, s_2, s_3).\) 
\end{proof}

\section{Julia set \(J_{\Pi}\) and stable sets}

In this section we continue to study Julia sets.  Set 
\(f : =P_{A_3}^d(z_1,z_2,z_3)\). From  Proposition \ref{pro:A}  we know \(f\) is a regular polynomial endomorphism.  So \(f\) extends continuously and holomorphically to \({\mathbb P^3} \), still denoted by \(f\).   We will study the Julia sets \(J_2(f)\),  \(J_1(f)\) and  \(J_2(f_{\Pi})\),  where      \(f_{\Pi}\)  denotes the restriction of \(f\)  to the hyperplane \(\Pi\) at infinity.  Note that  \(\Pi\) is completely invariant under \(f\).

 The B\(\ddot{o}\)ttcher coordinate is useful in holomorphic dynamics in one complex variable.  We try to construct analogous maps to the  B\(\ddot{o}\)ttcher coordinate.\\  
Let \(f_h\) denote the homogeneous part of degree \(d\) of  \(f(z_1,z_2,z_3)\). 
\[\mbox{Set} \quad \Phi_2(x, y, z) = (x^2, x(y+\frac 1y)/z, 1/z^2).\]
\begin{pro}\label{pro:I}
\(f\) and \(f_h\) satisfy the following commutative diagram.
\begin{equation}
\begin {array}{cccc}
  (z_1,z_2,z_3) & \xrightarrow{f} &  (z_1^{(d)}, z_2^{(d)}, z_3^{(d)}) & \\
    \uparrow{\Phi _{1} \enskip} &  &  \uparrow{\Phi _{1} \enskip} \\
   (t_1, t_2, t_3) &  \xrightarrow{ }& (t_1^d, t_2^d, t_3^d)   & \\
      \uparrow{} &  & \uparrow{} \\
   (\sqrt{t_1}, \sqrt{t_2}, \sqrt{t_3}) &  \xrightarrow{  }& (\sqrt{t_1}^d, \sqrt{t_2}^d, \sqrt{t_3}^d)   \\
   \downarrow{\Phi _{2} \enskip} &  & \downarrow{\Phi _{2} \enskip}  & \\
   (t_1, \frac{\sqrt{t_1}}{\sqrt{t_3}}(\sqrt{t_2}+\frac 1{\sqrt{t_2}}),\frac 1{t_3}) &  \xrightarrow{f_h }& (t_1^d, (\frac{\sqrt{t_1}}{\sqrt{t_3}})^d(\sqrt{t_2}^d +\frac 1{\sqrt{t_2}^d}), \frac 1{t_3^d}) &
\end{array}
\end{equation}
where  \(t_j  \in \mathbb{C}\setminus \{0\}\)  and  \( \sqrt{t_1}, \sqrt{t_2}, \sqrt{t_3}\) are arbitrary branches and
 
\begin{equation}
\begin{split}
z_1^{(d)} = t_1^d + t_2^d + t_3^d + \frac {1}{t_1^dt_2^dt_3^d}, \qquad \\
z_2^{(d)} = t_1^dt_2^d + t_1^d t_3^d + t_2^d t_3^d + \frac 1{t_1^dt_2^d} + \frac 1{t_1^dt_3^d}+ \frac 1{t_2^d t_3^d},\\
z_3^{(d)}= \frac 1{t_1^d} + \frac 1{t_2^d} + \frac1{t_3^d} + t_1^dt_2^d t_3^d.\qquad \qquad
\end{split} 
\end{equation}
\end{pro}
\begin{proof}
The upper-half of the commutative diagram is shown in (1.2).  We prove the lower-half of the diagram by induction on \(d\).  If  \(d = 2\) or \(3\), we can directly prove that the diagrams is commutative. The function \(f_h\) is considered in the proof of Proposition \ref{pro:A}.
\[f_h(x, y, z) = (x^d, h_2^{(d)}(x, y, z) , z^d).\]  
\[\mbox{Set} \quad \Phi_2(\sqrt{t_1}, \sqrt{t_2}, \sqrt{t_3}) = (x, y, z).\]  
\[\mbox{Then} \quad h_2^{(d+2)} \circ \Phi_2 = yh_2^{(d+1)} \circ \Phi_2 - xzh_2^{(d)} \circ \Phi_2.\] 
 Hence the diagram is commutative for any \(d\).
\end{proof}

 We use the definitions and notations in \cite{BJ}.
Let \(\Pi : = {\mathbb P^3} - {\mathbb C^3}\),  the plane at infinity.  It is isomorphic to \({\mathbb P^2}\).   
Clearly,  \(\Pi\) is completely invariant.
Let \(f_{\Pi}\) denote the restriction of \(f\) to \(\Pi\).
We may define the current  \(T_{\Pi} : = T \vert _{\Pi}\)
as the slice current.
Set
\[\mu _{\Pi} : = T_{\Pi}^2 \quad \mbox{and} \quad J_2(f_{\Pi})  : = supp(\mu_{\Pi}).\]
Bedford and Jonsson \cite{BJ} use the symbol \(J_{\Pi}\) for \(J_2(f_{\Pi})\).  We have the following statements for \(J_{\Pi}\) and \(\mu_{\Pi}\).\\

\begin{theorem}  
(1) The Julia set  \(J_2(f_{\Pi}) \) is a M\(\ddot{o} \)bius strip \(\mathcal{M}\). 
\[ \mathcal{M}  = \{(e^{\theta i}, xe^{\frac \theta 2 i }) :  0 \leq \theta < 2\pi, \enskip -2 \leq x \leq 2\}.\]
 (2) The maximal entropy measure \(\mu = \mu_{\Pi}\) is given by 　
 \[\sigma_*(\mu ) = \frac{d\theta }{2\pi } \quad \mbox{on} \quad \{e^{i\theta } : 0 \leq \theta < 2\pi \}\enskip \mbox{in the} \enskip \xi -\mbox{plane} ,\]
 \[\mu (\cdot \mid \sigma ^{-1}(\xi )) = \frac 1\pi \frac{dx}{\sqrt{4-x^2}}\quad \mbox{on} \quad \{xe^{\frac \theta  2 i}:  -2\leq x \leq 2\}.\]
 Here \({f}_{\Pi} (z_1:z_2:z_3) = f_{\Pi}(\xi :\eta :1), \quad \mbox{and} \quad \sigma (\xi ,\eta ) = \xi \) .\\
(3)  The Lyapunov exponents    of   \( f_{\Pi} \) with respect  to \(\mu\) are given by      
\(\lambda _1 = \lambda _2 =  \log d\).  
\end{theorem}

To prove this theorem we use Jonsson's results in \cite{J}.  In \cite{J}, Jonsson study polynomial skew product maps on ${\mathbb C}^2$.  A polynomial skew product of ${\mathbb C}^2$ of degree \(d \ge 2\) is a map of the form \(f(z,w) = (p(z), q(z,w))\),  where \(p\) and \(q\) are polynomials of degree \(d\).  Let \(G_p(z)\) be the Green function of \(p\) and \(G(z, w)\) be the Green function of \(f\) on ${\mathbb C}^2$.  
\[\mbox {Set} \quad K_p : = \{G_p = 0\} \quad \mbox {and} \quad J_p : = \partial K_p.\]
Define \(G_z(w) : = G(z,w) - G_p(z)\).
\[\mbox {Let} \quad K_z : = \{G_z = 0\} \quad \mbox {and} \quad J_z : = \partial K_z.\]
{\bf Proof of Theorem 3.2} 

(1):  Let \(\pi\) be the projection from ${\mathbb C}^3 - \{0\} $  to \(\Pi\).  Then \(\pi \circ f_h = f_{\Pi} \circ \pi\). 
\[\mbox {Since} \quad f_h(z,w,v) = (z^d, h_2^{(d)}(z,w,v), v^d),  \quad \mbox{it follows that}\]
\[ f_{\Pi}(z:w:v) = (z^d: h_2^{(d)}(z,w,v): v^d).\]
Case 1 : \(v = 0\).  The line \(\{v = 0\}\)  at infinity in \(\Pi\)  is an attracting set of \(f_{\Pi}(z:w:v)\).  Hence   
 there is a neighborhood of \(\{v = 0\}\) which does not have any repelling periodic point of \(f_{\Pi}\).   Therefore 
\[ \{v = 0\} \cap J_2(f_{\Pi}) =  
\emptyset.\]
\[\mbox {Case 2} :  v \ne 0.\quad \mbox{Then} \quad f_{\Pi}(z:w:1) = (z^d : h_2^{(d)}(z,w,1) : 1) \qquad\qquad\qquad\qquad\qquad\qquad\]
and so we consider a polynomial skew product on ${\mathbb C}^2$, still denoted by \(f_{\Pi}\),
\[f_{\Pi}(z,w) = (z^d, h_2^{(d)}(z,w,1) ).\]
Set \(z = t_1\)  and \(w = \sqrt {t_1}( \sqrt {t_2} + \frac 1{ \sqrt {t_2}}).\)  Then from (3.1) we see that
\begin{equation}
 f_{\Pi}(t_1, \sqrt {t_1}( \sqrt {t_2} + \frac 1{ \sqrt {t_2}})) = (t_1^d,  \sqrt {t_1}^d( \sqrt {t_2}^d + \frac 1{ \sqrt {t_2}^d})).
\end{equation}
We use Jonsson's results.  In our case \(p(z) = z^d\)  and so \(J_p = \{\mid z \mid = 1\}.\)  Hence, we may assume \enskip \(z = t_1 \ne 0\) .   To use Corollary 4.4 in \cite{J}, we consider \(K_a\) for any \(a = e^{i\theta} \in J_p\).  Let  \(t_1 = e^{i\theta}\).  Since \(G_p(a) = 0,\)  we have \(G_a(w) = G(a, w)\), where 
\[G(a,w) = \lim_{n\to \infty}d^{-n}\log^+\mid f_{\Pi}^n(a, w) \mid.\]
From (3.3) and the definition of \(K_a\), we see that \(w \in K_a\)  if and only if  \(w = e^{i\theta/2}(e^{i\phi} + e^{-i\phi})\) with  \(0 \le \phi \le 2\pi\).
\[\mbox {Hence} \quad K_a = \{2\cos \phi e^{\frac{i\theta}2} : 0 \le \phi \le 2\pi\}. \quad \mbox {Therefore} \]
\[ J_a = \partial K_a = K_a.\]
By Corollary 4.4 in \cite{J},  we conclude that 
\[J_2(f_{\Pi}) = \overline{\cup_{a \in J_p}\{a\}\times J_a} = \{ (e^{i\theta}, 2\cos\phi  e^{i\theta/2}) : 0 \le \theta \le 2\pi, 0 \le \phi \le \pi\}.\]

(2):  To prove the assertion (2), we use Theorem 4.2 in\cite{J}.  The action of \(\mu\) on a test function \(\varphi\) is given by
\[\int \varphi \mu = \int (\int \varphi(z,w)\mu_z(w))\mu_p(z).\]
\[\mbox{Here} \quad \mu_p : = \frac 1{2\pi}dd^cG_p \quad \mbox{and} \quad  \mu_z : = \frac 1{2\pi}dd^cG_z .\]
Since \(p(z) = z^d\),  it follows that \(\mu_p = \frac 1{2\pi}d\theta\)  and supp\((\mu_p)\) is the unit circle \(S^1\).  We will compute 
\[G_z(w) : = G(z,w) - G_p(z) \quad \mbox{and} \quad \mu_z \quad \mbox{for} \quad z \in S^1.\]
Let \(a = e^{i\theta}\).

As before we set \(z = t_1 = a\)  and  \(w = \sqrt {t_1}( \sqrt {t_2} + \frac 1{ \sqrt {t_2}}).\)    From  (3.3),  we have
\[\mid f_{\Pi}^n(a, w)\mid^2 = \mid a^{d^n}\mid^2 + \mid( \sqrt {a})^{d^n}( \sqrt {t_2}^{d^n} + \frac 1{ \sqrt {t_2}^{d^n}})\mid ^2  = 1 + \mid \sqrt {t_2}^{d^n} + \frac 1{ \sqrt {t_2}^{d^n}})\mid^2 .\] 
Hence
\begin{equation*}
\begin{split}
 G(a,w) = \lim_{n\to \infty}\frac 1{2d^n} \log(1+ \mid \sqrt {t_2}^{d^n} + \frac 1{ \sqrt {t_2}^{d^n}})\mid^2)\\
=\lim_{n\to \infty}\frac 1{2d^n} \log^+ \mid \sqrt {t_2}^{d^n} + \frac 1{ \sqrt {t_2}^{d^n}}\mid^2\quad\\
= \lim_{n\to \infty}\frac 1{d^n} \log^+ \mid \sqrt {t_2}^{d^n} + \frac 1{ \sqrt {t_2}^{d^n}}\mid\qquad\quad\\
= \lim_{n\to \infty}\frac 1{d^n} \log^+ \mid T_d^n(u) \mid  \qquad\qquad\\
= G_T(u) \qquad\qquad\qquad\qquad\qquad\qquad
\end{split}
\end{equation*}
Here \(T_d(u)\) is the Chebyshev polynomial of degree \(d\) of a single variable   \(u = ( \sqrt {t_2} + \frac 1{ \sqrt {t_2}})\)  and \(G_T(u)\) is the Green function of   \(T_d(u)\).

Since \(w = e^{\frac{i\theta}2} u\) and \(G_T(u) = G(a,w) = G_a(w)\),  we have 
\[\frac{\partial^2}{\partial u\partial\bar{u}} G_T(u) = e^{-\frac {i\theta}2} \cdot e^{\frac{i\theta}2} \frac{\partial^2}{\partial w\partial\bar{w}}G(e^{i\theta}, w) = \frac{\partial^2}{\partial w\partial\bar{w}}G_a( w).\]
It is known in \cite{UN}  that the maximal entropy measure \((1/2\pi)dd^cG_T(u)\)  of \(T_d(u)\) is equal to \(\frac 1{\pi}\frac{du_1}{\sqrt{4-u_1}}\)  supported on the segment \(\{u_1 : -2 \le u_1 \le 2\},\)  where \( u_1 =Re(u) .\)\\
Hence the current \(\mu_a\) is given by  
\[ \frac 1\pi \frac{dx}{\sqrt{4-x^2}}\quad \mbox{on} \quad \{xe^{\frac \theta  2 i}:  -2\leq x \leq 2\}.\]

(3): We have proved that \(J_p\) is connected and  each \(J_a\)  is connected for all  \(a \in J_p\).  Hence from Theorem 6.5 in \cite{J}  we have
\[\lambda_1 = \lambda_2 = \log d.  \qquad \qquad \Box\]

We continue to study Julia sets.  We consider orbits of \(f\) and classify all the points of \({\mathbb C}^3\) into four categories.  We begin with finding invariant sets of \(f\) in \({\mathbb P}^3\).  We have already two invariant sets \(K(f)\) and \(J_2(f_{\Pi})\).  
Besides these sets, there are two circles :
\[ S_1 : = \{(1:e^{i\theta} : 0 : 0) : 0 \le \theta < 2\pi\} , \quad
S_2 : = \{(0:e^{i\theta} : 1 : 0) : 0 \le \theta < 2\pi\},\]
and three attracting fixed points  :
\[P_1 = (1 : 0 : 0 : 0), \quad P_2 = (0 : 1 : 0 : 0), \quad P_3 = (0 : 0 : 1 : 0).\]
We define the stable set of an invariant set \(X\)  by 
\[W^s(X,f) = \{x \in  {\mathbb P^3} : d(f^nx, X) \to 0 \quad \mbox{as}\quad  n \to \infty\}.\]
Then we have the following proposition.
\begin{pro}\label{pro:2}
Let \(a,b,c,d\) be a permutation of the set  \(\{\mid t_1\mid, \mid t_2\mid, \mid t_3 \mid, \mid t_4 \mid\}, \quad \mbox{where} \quad t_4  = \frac 1{t_1t_2t_3}\).\\
(1) If \(a = b = c = d =1,\)  then  \(\Phi_1(t_1, t_2, t_3) \in K(f)\).\\
 (2) If \(a > b = c  = 1 > d = \frac 1a,\)  then  \(\Phi_1(t_1, t_2, t_3) \in W^s(J_2(f_{\Pi}), f)\).\\
 (3) If \(a > b =  1 > c \ge d \)  or \(a \ge b > c = 1 > d  \),  then  \(\Phi_1(t_1, t_2, t_3) \in W^s(S_1 \cup S_2, f)\).\\
(4) If \((a-1)(b-1)(c-1)(d-1) \ne 0\), \enskip  then  \(\Phi_1(t_1, t_2, t_3) \in W^s(P_1\cup P_2\cup P_3, f)\).\\
\end{pro}
\begin{proof}\quad (1):
The assertion (1) is already shown in Proposition 2.2.  

 (2):  Let \(r_j = \mid t_j \mid, \quad (j = 1, 2, 3, 4)\). \quad We assume that 
\[r_1 = r, \quad r_3 = \frac 1r, \quad r_2 = r_4 =1, \quad r > 1. \quad \mbox  {Then} \qquad \qquad \qquad\]
\[z_1 = re^{i\alpha  } + e^{i\beta }+\frac {e^{i\gamma}}r +  e^{i(-\alpha -\beta -\gamma)},\qquad \qquad \qquad \qquad \qquad\]
\[z_2 = re^{i(\alpha+\beta)  } + e^{i(\alpha+\gamma ) } +re^{i(-\gamma -\beta )}+ \frac 1r e^{i(\beta +\gamma)} + e^{i( -\alpha-\gamma )} + \frac 1r e^{-i(\alpha +\beta )},\]
\[z_3 =\frac 1re^{-i\alpha  } + e^{-i\beta } + re^{-i\gamma} + e^{i(\alpha +\beta +\gamma )}.\qquad \qquad \qquad \qquad\]
The dominant terms of \enskip \(z_1, z_2, z_3\) \enskip are \enskip \(re^{i\alpha}, \enskip re^{i(\alpha+\beta)} + re^{i(-\beta-\gamma)}, \enskip re^{-i\gamma}, \enskip\) respectively.  Then for large \(n\),
\[f^n(z_1 : z_2 : z_3 : 1) \simeq (\exp(i\alpha d^n) : \exp(i(\alpha+\beta) d^n) + \exp(-i(\beta+\gamma) d^n) : \exp(-i\gamma d^n) : \frac 1{r^{d^n}})\]
\[= (\exp(i(\alpha+\gamma) d^n) : \exp(i(\alpha+\gamma) \frac{d^n}2)\cdot 2\cos((\frac{\alpha+\gamma}2 + \beta)d^n)  : 1 : \exp (i\gamma d^n) /r^{d^n}).\]
Hence
\[(z_1 : z_2 : z_3 : 1) \in W^s(\{(e^{i\sigma} : 2\cos\tau e^{\frac{i\sigma}2} : 1 : 0) : 0 \le \sigma < 2\pi,   0 \le \tau < \pi\},  f) = W^s(J_2(f_{\Pi}), f).\]
Then the assertion (2) follows.

(3):  We assume that \quad \(r_1 \ge r_2 \ge r_3\). \enskip If \enskip \(a > b = 1 > c \ge d,\) \enskip  then there are four cases :
\[(i)\enskip r_4 > r_1 = 1 > r_2 \ge r_3, \quad (ii) \enskip r_1 > r_4 = 1 > r_2 \ge r_3, \]
\[(iii) \enskip r_1 > r_2 = 1 > r_4 \ge r_3, \quad (iv)\enskip r_1 > r_2 = 1 > r_3 \ge r_4. \]
\[\mbox{Let} \quad M(z_1) : = \max\{r_1, r_2, r_3, r_4\}, \qquad \qquad\]
\[M(z_2) : = \max\{r_1r_2, r_1r_3, r_1r_4, r_2r_3, r_2r_4, r_3r_4\},\]
\[ M(z_3) : = \max\{\frac 1{r_1}, \frac 1{r_2}, \frac 1{r_3}, \frac 1{r_4}\}. \qquad \qquad\]

Let \(dom(z_j)\) be the set of the maximum elements that are equal to \(M(z_j)\).\\
 Case (i).  Then  \(dom(z_1) = \{r_4\}, \quad  dom(z_2) = \{r_1r_4\}, \quad  M(z_3) = \frac 1{r_3}\).\quad
Hence  \(M(z_1) =  M(z_2) >  M(z_3)\). \\
For other cases, we can show that \(dom(z_1)\) and \(dom(z_2)\) are singletons and that  \(M(z_1) =  M(z_2) >  M(z_3)\).  Hence if we set \(r: =  M(z_1) =  M(z_2) \),   then
\[f^n(z_1 : z_2 : z_3 : 1) \simeq (\exp(i\sigma d^n) : \exp(i \tau d^n)  : \varepsilon_n : \frac 1{r^{d^n}}), \quad \mbox{with}\enskip \varepsilon_n \to 0 \enskip( n \to \infty).\]
Hence
\[(z_1 : z_2 : z_3 : 1) \in W^s(\{(1 : e^{i\theta}  : 0 : 0) : 0 \le \theta < 2\pi\},   f) .\]

Similarly we can prove that  if \enskip \(a \ge b > c = 1 > d, \)  \enskip then
\[(z_1 : z_2 : z_3 : 1) \in W^s(\{(0 : e^{i\theta}  : 1 : 0) : 0 \le \theta < 2\pi\},   f) .\]
Then the assertion (3) follows.\\
(4):  If\enskip \((a-1)(b-1)(c-1)(d-1) \ne 0,\)    there are three cases:
\[(i) \enskip a > 1 > b \ge c \ge d, \quad (ii) \enskip a \ge b > 1 > c \ge d, \quad (iii)\enskip a \ge b \ge c > 1 > d.\]
Case (i). \quad  Then we see that \(M(z_1) >  M(z_2) ,  M(z_3)\) and \(dom(z_1)\) is a singleton.    
  \[{\mbox Hence} \quad (z_1 : z_2 : z_3 : 1) \in W^s((1 : 0  : 0 : 0), f) .\]
Case (ii). \quad  Then we see that \(M(z_2) >  M(z_1) ,  M(z_3)\)  and \(dom(z_2)\) is a singleton.  
  \[{\mbox Hence} \quad(z_1 : z_2 : z_3 : 1) \in W^s((0 : 1  : 0 : 0), f) .\]
 Case (iii). \quad  Then we see that \(M(z_3) >  M(z_1) ,  M(z_2)\) and \(dom(z_3)\) is a singleton.  
  \[{\mbox Hence} \quad(z_1 : z_2 : z_3 : 1) \in W^s((0 : 0  : 1 : 0), f) .\]
\end{proof}

\section{ Julia sets  \(J_1\),  \(J_2\) and external rays}

External rays for holomorphic endomorphisms of  \( {\mathbb P^k}\)  are introduced  by Bedford and Jonsson \cite{BJ}.  We review some results in  \cite{BJ}.  Global stable manifolds at each point of a in   \( J_{\Pi}\)  is defined by  
\[W^s(a) = \{x \in  {\mathbb P^k} : d(f^jx, f^ja) \to 0 \quad \mbox{as}\quad  j \to \infty\}.\]
Note that \(W^s(a)\) contains all the local stable manifold \(W^s_{loc}(b)\) for \(b \in J_{\Pi}\)  with  \(f_{\Pi}^nb = f_{\Pi}^na, \quad n \ge 0.\)\quad
Divide \(W^s(a)\) into stable disks \(W_a\).  Let \(\mathcal{E}_a\) denote the set of all gradient lines in \(W_a\) and let the set \(\mathcal{E}\) of external rays be the union of all \(\mathcal{E}_a\).  Note that \(f\) maps gradient lines to gradient lines.  

In this paper, using  'B\(\ddot o\)ttcher coordinate' we construct global external rays.  We consider 
\(\Phi _1(re^{i\alpha }, e^{i\beta }, \frac 1r e^{i\gamma }) ; \)
\begin{equation}
\begin{split}
z_1 = re^{i\alpha  } + e^{i\beta }+\frac {e^{i\gamma}}r +  e^{i(-\alpha -\beta -\gamma)},\qquad \qquad \qquad \qquad \qquad \qquad\\
z_2 = re^{i(\alpha+\beta)  } + e^{i(\alpha+\gamma ) } +re^{i(-\gamma -\beta )}+ \frac 1r e^{i(\beta +\gamma)} + e^{i( -\alpha-\gamma )} + \frac 1r e^{-i(\alpha +\beta )},\\
z_3 =\frac 1re^{-i\alpha  } + e^{-i\beta } + re^{-i\gamma} + e^{i(\alpha +\beta +\gamma )}.\qquad \qquad \qquad \qquad \qquad \qquad\\
\end{split}
\end{equation}
Let 　\(R(\alpha, \beta, \gamma ; r)\) denote this point 
\(\quad \Phi _1(re^{i\alpha }, e^{i\beta }, \frac 1r e^{i\gamma })\quad \mbox{in} \quad \mathbb {P}^3.\)
Then using an argument similar to the proof of Proposition 3.3 (2) , we can prove that 
\[ R(\alpha, \beta, \gamma ; \infty)  = (e^{i(\alpha+\gamma)  }  : (2\cos(\frac{\alpha +\gamma }2 + \beta ))e^{i\frac{\alpha +\gamma }2} : 1: 0)  \in  J_\Pi, \qquad\]
where
\[ R(\alpha, \beta, \gamma ; \infty) : = \lim_{r \to \infty}
R(\alpha, \beta, \gamma ; r).\]
Clearly,  \(\quad R(\alpha, \beta, \gamma ; 1) \in K(f)\) and  \(R(\alpha, \beta, \gamma ; r) = R(\alpha,  - \alpha -\beta - \gamma, \gamma  ; r)\). 

 Define an {\it external ray} by \(R(\alpha, \beta, \gamma )  :   = \{R(\alpha, \beta, \gamma ; r) : r  > 1\}.\) 

(External rays of \(f_h\) are given by 
\(\quad \{\Phi _2(re^{i\alpha }, e^{i\beta }, \frac 1r e^{i\gamma }) : r > 1\}\).)   

Clearly,
 \[ f(R(\alpha, \beta, \gamma ; r)) =  R(d\alpha, d\beta, d\gamma ; r^d).\]   Then
\[ f(R(\alpha, \beta, \gamma)) =  R(d\alpha, d\beta, d\gamma),\]
and    
\[\mbox{if} \enskip r > 1, \enskip \lim_{n \to \infty}f^n(R(\alpha, \beta, \gamma ; r)) \in J_{\Pi}.\] 
We  set
\[D(\alpha+\gamma , \beta ) : = \bigcup _{0\leq \theta < 2\pi }R(\alpha - \theta , \beta , \gamma + \theta ).\] 
By the above equality,  we have \enskip \(f(D(\alpha+\gamma , \beta) )  = D(d(\alpha+\gamma), d\beta )\).　
Next lemma shows that  \(D(\alpha+\gamma , \beta ) \) is a stable disk passing through 
\(R(\alpha, \beta, \gamma ; \infty)\).
\begin{lemma} \label{lemma:4-1}
\(D(\alpha+\gamma , \beta ) \subset W^s( R(\alpha, \beta, \gamma ; \infty))\) .
\end{lemma}
\begin{proof}
Let \((z_1,  z_2,  z_3 )\)  be any point of  \(R(\alpha - \theta , \beta , \gamma + \theta ).\)   The dominant terms of　\(z_1, z_2\)  and   \(z_3 \) \quad are \enskip \(re^{i(\alpha-\theta)}, \enskip re^{i(\alpha+\beta-\theta)} + re^{i(-\beta-\gamma-\theta)} \)\enskip and \enskip \(re^{-i(\gamma+\theta)} ,\enskip\) respectively.
As in the proof of Proposition 3.3(2), we can prove that
\[f^n(z_1 : z_2 : z_3 : 1) \simeq  (\exp(i(\alpha+\gamma) d^n) : \exp(i(\alpha+\gamma) \frac{d^n}2)\cdot 2\cos((\frac{\alpha+\gamma}2 + \beta)d^n)  : 1 : \exp (i(\gamma+\theta) d^n) /r^{d^n}).\]
On the other hand, by Proposition 3.1, we have
\[f^n_{\Pi}(R(\alpha, \beta, \gamma ; \infty)) = f^n_{\Pi}(e^{i(\alpha+\gamma)} : e^{\frac{\alpha+\gamma}2 i}  (e^{(\frac{\alpha+\gamma}2+\beta) i} + e^{-(\frac{\alpha+\gamma}2+\beta) i}) : 1 : 0) \] 
\[= (\exp(i(\alpha+\gamma) d^n) : \exp(i(\alpha+\gamma) \frac{d^n}2)\cdot 2\cos((\frac{\alpha+\gamma}2 + \beta)d^n)  : 1 : 0).\]
Then the lemma follows.
\end{proof}
From Proposition 3.3, we deduce that the set 
\(\{D(\alpha+\gamma, \beta)\}\)  forms a foliation of 
\( W^s(J_{\Pi},  f)\).

Now we will determine the Julia sets  \(J_2(f)\)  and   \(J_1(f)\).  Using a result in \cite{BJ}  we will determine  \(J_2(f)\).  Corollary 8.5 of \cite{BJ}  reads as follows.    For almost every \(a \in J_{\Pi}\), we have \(\overline{W^s(a)} = supp(T^{k-1} \, \llcorner \, \{G > 0\}).\)  Here \(G\) is the Green function of \(f\).  

   Using this and Proposition 3.3, we have the following.  Let \(F(f)\) denote the Fatou set of \(f\).

\begin{theorem} 
\( {\mathbb P^3}\) decomposes into the following sets; \\
(1) \(J_3(f) = K(f)\),\\ 
(2)  \(J_2(f)\setminus J_3(f) = W^s(J_2(f_{\Pi}), f) = \cup D(\alpha+\beta, \beta),\)\\
(3)  \(J_1(f)\setminus J_2(f) = W^s(S_1 \cup S_2 , f)\),\\
(4) \(F(f) = W^s(P_1\cup P_2\cup P_3, f)\).
\end{theorem}
\begin{proof}\quad (1): The assertion (1) is shown in Theorem 2.7 (1). 

(2):  To prove (2),  we need Corollary 8.5 of \cite{BJ}.  We know in Theorem 3.2 that 
\[J_2(f_{\Pi}) = \mathcal{M} = \{(e^{i\theta}, xe^{\frac{i\theta}2}) :  0 \le \theta < 2{\pi},  -2 \le x \le 2\}.\]
And  the maximal entropy measure \(\mu_{\Pi}\) is given there.  By Corollary 8.5 of \cite{BJ},  we see that there is an element a in \(\mathcal{M}\)  such that
\begin{equation}
\overline{W^s(a)} = supp(T^{2} \, \llcorner \, \{G > 0\}).
\end{equation} 
Set  \(a = (e^{i\theta}, xe^{\frac {\theta}2 i})\).  

We claim that
\begin{equation}
J_2(f_{\Pi}) = \overline{\bigcup_n f^{-n}_{\Pi}(f^n_{\Pi}(a))} .
\end{equation}
To see this, we know in the proof of  Theorem 3.2 that 
\[f_{\Pi}(z, w) = (z^d, h_2^{(d)} (z, w, 1)).\]
Since \(e^{i\theta} \in J_p\)  with  \(p(z) = z^d\),  \(\bigcup_n p^{-n}( e^{i\theta} )\) is dense in \(J_p = S^1\).  Also  the set  \(\bigcup_n p^{-n}(p^n( e^{i\theta}) )\) is dense in \(J_p\) . 
From Theorem 3.2 (2) we know that on the fibers  \(\{\sigma^{-1}(z) : z \in \bigcup_n p^{-n}(p^n( e^{i\theta}) )\}, \quad h_2^{(d)}\)  acts as the Chebyshev map \(T_d\).  Then (4.3) follows.

For any  \(c \in \overline{\bigcup_n f^{-n}_{\Pi}(f^n_{\Pi}(a))} ,\)  there is a sequence \(\{b_m\}\)  with 
\(b_m \in {\bigcup_n f^{-n}_{\Pi}(f^n_{\Pi}(a))} \) such that  \(b_m \to c\)  as \(m \to \infty\).   Since \(b_m \in W^s(a) \),  it follows that \enskip  \(c \in \overline{W^s(a)}\).  
Set \(c =  R(\alpha, \beta, \gamma ; \infty)\)  and \(b_m =  R(\alpha_m, \beta_m, \gamma_m ; \infty)\) .
Then we have \((\alpha_m+\gamma_m , \beta_m ) \to (\alpha+\gamma , \beta )\).

We claim that
\begin{equation}
D(\alpha+\gamma , \beta ) \subset \overline{W^s( a)}.
\end{equation}
Indeed.  We have shown that  the center \(R(\alpha, \beta, \gamma ; \infty)\)  of the disk \enskip \(D(\alpha+\gamma , \beta )\) is in  \(\overline{W^s( a)}\).  For any point  \(R(\alpha - \theta , \beta , \gamma + \theta ; r )\)  in  \(D(\alpha+\gamma , \beta )\),  we can select a sequence \(\{R(\alpha - \theta , \beta_m , \alpha_m + \gamma_m - \alpha + \theta ; r)\}\)  such that  
 \[R(\alpha - \theta , \beta_m , \alpha_m + \gamma_m - \alpha + \theta ; r) \to R(\alpha - \theta , \beta , \gamma + \theta ; r )\quad \mbox{as} \quad m \to \infty.\]
Hence from Lemma \ref{lemma:4-1}, we have 
\[R(\alpha - \theta , \beta_m , \alpha_m + \gamma_m - \alpha + \theta ; r) \in D(\alpha_m+\gamma_m , \beta_m ) \subset W^s(R(\alpha_m, \beta_m, \gamma_m ; \infty) ) .\]
Since
\[W^s(R(\alpha_m, \beta_m, \gamma_m ; \infty) )  = W^s(b_m) = W^s(a),\]
it follows that \(R(\alpha - \theta , \beta_m , \alpha_m + \gamma_m - \alpha + \theta ; r) \in W^s(a)\).
Then  \(R(\alpha - \theta , \beta , \gamma + \theta ; r ) \in \overline{W^s(a)}\).
Therefore (4.4) follows.  

Hence from (4.3) we deduce that  
\begin{equation}
\bigcup_{\alpha+\gamma , \beta}D(\alpha+\gamma , \beta ) \subset \overline{W^s( a)}.
\end{equation}

Conversely we claim that
\begin{equation}
\bigcup_{\alpha+\gamma , \beta}D(\alpha+\gamma , \beta ) \supset {W^s( a)}.
\end{equation}
Indeed.  In the first place we consider any element  \(b\) of  \(W^s( a) \cap \Pi\).  From the proof of Theorem 3.2 (1),  we may assume that  \(b = (z : w : v)\)  with  \(v \ne 0\).  By the case 2 of the proof of Theorem 3.2 (1),  we  see that  \(b \in J_{\Pi}\).  Then \(b \in \bigcup_{\alpha+\gamma , \beta}D(\alpha+\gamma , \beta ) \).

Next we assume that \((z_1, z_2, z_3 )\)  is an element of \(W^s(a)\) in \({\mathbb C}^3\).  Then from Proposition 3.3, we see that \((z_1, z_2, z_3 )\) is written as  \(\Phi_1(t_1, t_2, t_3 )\)  in the assertion (2) of Proposition 3.3.  Then  we may set \quad \((z_1, z_2, z_3 )  = \Phi_1(re^{i\alpha}, e^{i\beta},  \frac 1r e^{i\gamma})\).
Hence
\((z_1, z_2, z_3 ) \in  R(\alpha, \beta, \gamma ) \subset D(\alpha+\gamma , \beta )\).
Then (4.6) follows.  

From (4.5) and (4.6), it follows that \enskip
\(\overline{W^s(a)} = \overline{\bigcup D(\alpha+\gamma , \beta )}.\)  \enskip The set \(\bigcup \overline{ D(\alpha+\gamma , \beta )}\) \enskip is a union of closed disks each of which is centered at a point of the M\(\ddot{o}\)bius strip.  Hence \enskip \(\bigcup \overline{ D(\alpha+\gamma , \beta )}\) \enskip is a closed set.  Then \enskip
\(\overline{\bigcup D(\alpha+\gamma , \beta )} =  \bigcup \overline{ D(\alpha+\gamma , \beta )}\).  \enskip Thus from (4.2) we have 
\[supp(T^2 \llcorner \{G > 0\}) = \overline{\bigcup_{\alpha+\gamma , \beta}  D(\alpha+\gamma , \beta )}= \bigcup_{\alpha+\gamma , \beta} \overline{ D(\alpha+\gamma , \beta )}.\]  

Set \enskip  \(A : = \{G > 0\}\).  Let \(U_1\) and \(U_2\) be the maximal open sets in which \enskip \(T^2 = 0\) \enskip and \enskip \(T^2 \llcorner A = 0\),  respectively.  Then \enskip \(supp T^2 = {\mathbb P}^3 \setminus U_1\)  and \enskip \(supp(T^2 \llcorner A ) = {\mathbb P}^3 \setminus U_2\).  Since \enskip \( K(f) = J_3  \subset suppT^2\) \enskip and \enskip \(\bigcup R(\alpha, \beta, \gamma ; 1) = K(f) \subset supp(T^2 \llcorner A )\),  \enskip we have 
\begin{equation}
U_i \cap K(f) = \emptyset , \enskip \enskip i = 1, 2.
\end{equation}
Let \(\psi\) be any 2-form of class \(C^{\infty}\) with compact support in \(U_1\).  Then by definition of   \(U_1\) and (4.7),  we have 
\[0 = <T^2, \psi> = <T^2, \psi \land \chi_A> = <T^2 \llcorner A,  \psi>,\]
where \(\chi_A\)  is a characteristic function of \(A\).  Then we have  \(U_1 \subset U_2\).  Similarly we can prove that   \(U_2 \subset U_1\).  Then it follows that \enskip
\(supp T^2  =  supp(T^2 \llcorner A )\).
Since \( K(f) = J_3(f)\),  we have
\(J_2(f) \backslash J_3(f) = \bigcup D(\alpha+\gamma , \beta )\).  The assertion (2) follows.  

 (3) and (4):  To prove (3) and (4) we note that if \(f\) is a holomorphic map from \({\mathbb P}^k\) to \({\mathbb P}^k\) , then the Julia set \(J_1(f)\) is the complement of the Fatou set of \(f\).  See Theorem 3.3.2 in\cite{S}.   

Note that \({\mathbb P}^k  = {\mathbb C}^3 \cup \Pi\).  In the first place we consider the set 
\({\mathbb C}^3\).  We have shown in Proposition 3.3 that    \({\mathbb C}^3\) decomposes into four categories.   Only the case (4) of Proposition 3.3  corresponds to the Fatou set \(F(f)\).

  Next we consider a decomposition of \(\Pi\).
We have shown in the proof of Theorem 3.2 that 
\[f_{\Pi}(z : w : v) = (z^d : h_2^{(d)} (z , w , v) : v^d).\]
Case 1 : \(v \ne 0\). \quad If  \(z = 0\),
\[f_{\Pi}(0 : w : v) = (0 : h_2^{(d)} (0 , w , v) : v^d).\]
From (2.4), we see that \( h_2^{(d)} (0, w, v) = w^d.\)
\[\mbox{Then if} \quad  \mid w \mid = \mid v \mid \quad \mbox{then} \quad (0 : w : v) \in S_2.\]
\[\mbox{If} \quad  \mid w \mid \ne \mid v \mid \quad \mbox{then} \quad (0 : w : v) \in W^s(P_2 \cup P_3, f_{\Pi}).\]
Next we assume that  \(z \ne 0\).  Then
\[f_{\Pi}(z, w) = (z^d, h_2^{(d)} (z, w, 1)).\]
We use the argument in the proof of Theorem 3.2.
Set \(z = t_1\)  and \(w = \sqrt {t_1}( \sqrt {t_2} + \frac 1{ \sqrt {t_2}}).\) And set  \(t_1 = r_1e^{i\sigma}\)  and \(t_2 = r_2e^{i\tau}\).
Then from (3.3) we have 
\[f^n_{\Pi}(z, w) = (r_1^{d^n}\exp(i\sigma d^n) , r_1^{d^n/2}\exp(i\sigma d^n/2)(r_2^{d^n/2}\exp(i\tau d^n/2) + r_2^{-d^n/2}\exp(-i\tau d^n/2))).\]
Hence if  \(r_1 = r_2 = 1,\)  then \((z, w)\) is an element of the M\(\ddot{o}\)bius strip \(\mathcal{M}\) .
\[\mbox{If} \quad  r_1 \ne 1 \quad \mbox{and } \quad (r_1 = r_2  \enskip \mbox{or} \enskip r_1r_2 = 1), \quad \mbox{then} \quad (z : w : 1) \in W^s(S_1 \cup S_2, f_{\Pi}).\]
\[\mbox{If} \quad r_1 \ne r_2  \enskip \mbox{and} \enskip r_1r_2 \ne 1,  \quad \mbox{then} \quad (z : w : 1) \in W^s(P_1 \cup P_2 \cup P_3, f_{\Pi}).\qquad\]
Case 2 : \(v = 0\).  Using an argument similar to the proof of the case \(z = 0\),  we have the following results. 
\[\mbox{If} \quad  \mid z \mid = \mid w \mid, \quad \mbox{then} \quad (z : w : 0) \in S_1.\]
\[\mbox{If} \quad  \mid z \mid \ne \mid w \mid, \quad \mbox{then} \quad (z : w : 0) \in W^s(P_1 \cup P_2, f_{\Pi}).\qquad\]
Now we combine the results on \({\mathbb C}^3\) and \(\Pi\).  Since the Fatou set of \(f\) is \( W^s(P_1 \cup P_2 \cup P_3, f)\),  the assertions  (3) and (4) follow.
\end{proof}

By direct computations, we can prove that \(J_1(f)\)  is a foliated space and leaves of the space are topological polydisks in  \({\bf \mathbb C^2}\).
\medskip

Next we consider  external rays in \(R_3 (=  \{(z_1,z_2,\bar z_1) : z_1 \in {\bf \mathbb C},  z_2 \in {\bf \mathbb R}\})\).  Recall that  any point\enskip \(R(\alpha, \beta, \gamma ; \infty) \in \mathcal{M}\) \enskip has a disk \(D(\alpha+\gamma, \beta)\) centered at itself.

 \begin{pro} \label{pro:Q}
\(\mbox{If} \quad  R(\alpha, \beta, \gamma) \subset  R_3 , \quad \mbox{then} \quad \alpha = \gamma.\)
 \(R(\alpha, \beta, \alpha)\) is a half-line and lands at a point of the astroidalhedron  \(\mathcal{A}\). 
Hence an external ray in  \(D(\alpha+\gamma, \beta)\) included in \(R_3\) is only the external ray  \(R(\frac{\alpha+\gamma}2, \beta, \frac{\alpha+\gamma}2)\).
\end{pro}
\begin{proof}
By (4.1), we have \enskip \(z_1 - \bar z_3 = (e^{i\alpha} - e^{i\gamma})(r - \frac 1r)\) .    If   \enskip  \(z_1 = \bar z_3\)  \enskip then  \enskip \(\alpha = \gamma\).    \enskip In the case,  \enskip \(R(\alpha, \beta, \alpha ; r)\)  \enskip is expressed as 
\begin{equation}
z_1 = (r + \frac 1r)e^{i\alpha} +  e^{i\beta} + e^{i(-2\alpha-\beta)}, \quad z_2 = 2(r + \frac 1r)\cos(\alpha+\beta) + 2\cos 2\alpha.
\end{equation}
Therefore \(R(\alpha, \beta, \alpha)\) is a half-line and lands at a point of the astroidalhedron  \(\mathcal{A}\). 
\end{proof}
 
We extend the half-line \(R(\alpha, \beta, \alpha)\) to the interior of  \(K(f)\).   In (4.8), we substitute  \(e^{i\theta}\)  for \(r\).  That is,
\begin{equation}
\begin{split}
z_1 = e^{i(\alpha+\theta)} + e^{i(\alpha-\theta)} +  e^{i\beta} + e^{i(-2\alpha-\beta)}, \\
 z_2 = 4\cos\theta\cos(\alpha+\beta) + 2\cos 2\alpha,  \quad 0 \le  \theta  < 2\pi.
\end{split}
\end{equation}
We call this the {\it internal ray} of    
\(R(\alpha, \beta, \alpha)\)  and denote it by   \(R_0(\alpha, \beta, \alpha)\) .  
 \begin{pro} \label{pro:R1}
Internal rays  \(R_0(\alpha, \beta, \alpha)\) are classified into two categories. \\
(1)  If  \(\alpha + \beta = 0\) \enskip or \enskip \(\alpha + \beta = \pi, \) \enskip  then the internal ray is a ruling of  \(\mathcal{A}\).\\ 
(2)  If  \(\alpha + \beta \ne 0,  \pi\), \enskip  then the internal ray \(R_0(\alpha, \beta, \alpha)\)  links two external rays \(R(\alpha, \beta, \alpha)\)  and  \(R(\alpha+ \pi, \beta,  \alpha+ \pi)\) . 
And the internal ray  touches the surface  \(\mathcal{A}\).   
\end{pro} 
\begin{proof}\quad (1):  If \enskip \(\alpha + \beta = 0,  \)  \enskip then  \[z_1 = 2\cos\theta e^{i\alpha} +  2e^{-i\alpha}, \enskip z_2 = 4\cos\theta + 2\cos 2\alpha, \enskip 0 \le  \theta  < 2\pi.\]   
Hence from (2.8) we know that this is a ruling of \(\mathcal{A}\).  The same holds for  \(\alpha + \beta = \pi \).
 
(2):  If \enskip \(\alpha + \beta \ne 0,  \pi \)  \enskip then the four  terms of \(z_1\)  in (4.9)  are distinct except for the cases 
\[\theta = 0, \quad \theta = \pi, \quad \theta = \pm(\alpha-\beta) \enskip \mbox{and}\quad \theta  = \pm (3\alpha+\beta).\] 
Then the internal ray is not included in  \(\mathcal{A}\)  and touches the surface at two points \(\quad \theta = \pm(\alpha-\beta) \enskip \mbox{and}\quad \theta  = \pm (3\alpha+\beta).\)   
\end{proof}
\begin{cor}\label{R2}
  The rulings of the astroidalhedron  are internal rays.  
\end{cor}
Next we study  'inscribed faces'  of \(\mathcal{A}\). Using the notations in Section 2,  we consider a face \(H\) in the natural domain \(R'\) in the space  \((\alpha, \beta, \gamma)\) defined by \quad
\(H : = \{\alpha = c\} \cap R'\), \enskip where \(c\) is a constant.   \(\varphi_1\)  is the map from \(R'\)  onto  \(K(f)\).  
\begin{pro} \label{pro:N1}
\(\varphi_1(H)\) is a face on the plane in the  \((p_1, p_2, q)\)  space  given by 
\[p_1 \cos c -  p_2 \sin c - q/2 = \cos 2c.\qquad \qquad \qquad \qquad\]
\end{pro}
\begin{proof}
By direct computations, we have this proposition.
\end{proof}

We denote four vertices of the polyhedron \(\partial R'\) by \enskip \(O(0, 0, 0)\), \enskip \(B_1(\pi/2, \pi/2, \pi/2)\), \enskip \(B_2(-\pi, \pi, \pi)\) \enskip and\enskip  \(B_3(-\pi/2, -\pi/2, 3\pi/2)\).
We consider the triangle \(\triangle OB_2B_3\).  It lies on the plane \enskip \(2\alpha + \beta + \gamma = 0.\) \enskip Set  \(L : = H \cap \triangle OB_2B_3\).
\enskip The line segment \(L\) is given by 
\enskip \(\{(c, \beta, -2c-\beta)\}.\)  The image of \(L\) under the transformation \(T\)  is a line segment which is parallel to the root \(\alpha_3\).  The image of \(\triangle OB_2B_3\) under \(\varphi_1\)  is a part of the surface \(\mathcal{A}\).
\begin{figure}[htbp]
\begin{tabular}{cc}
\begin{minipage}{0.45\hsize}
\begin{center}
\includegraphics[scale=0.56]{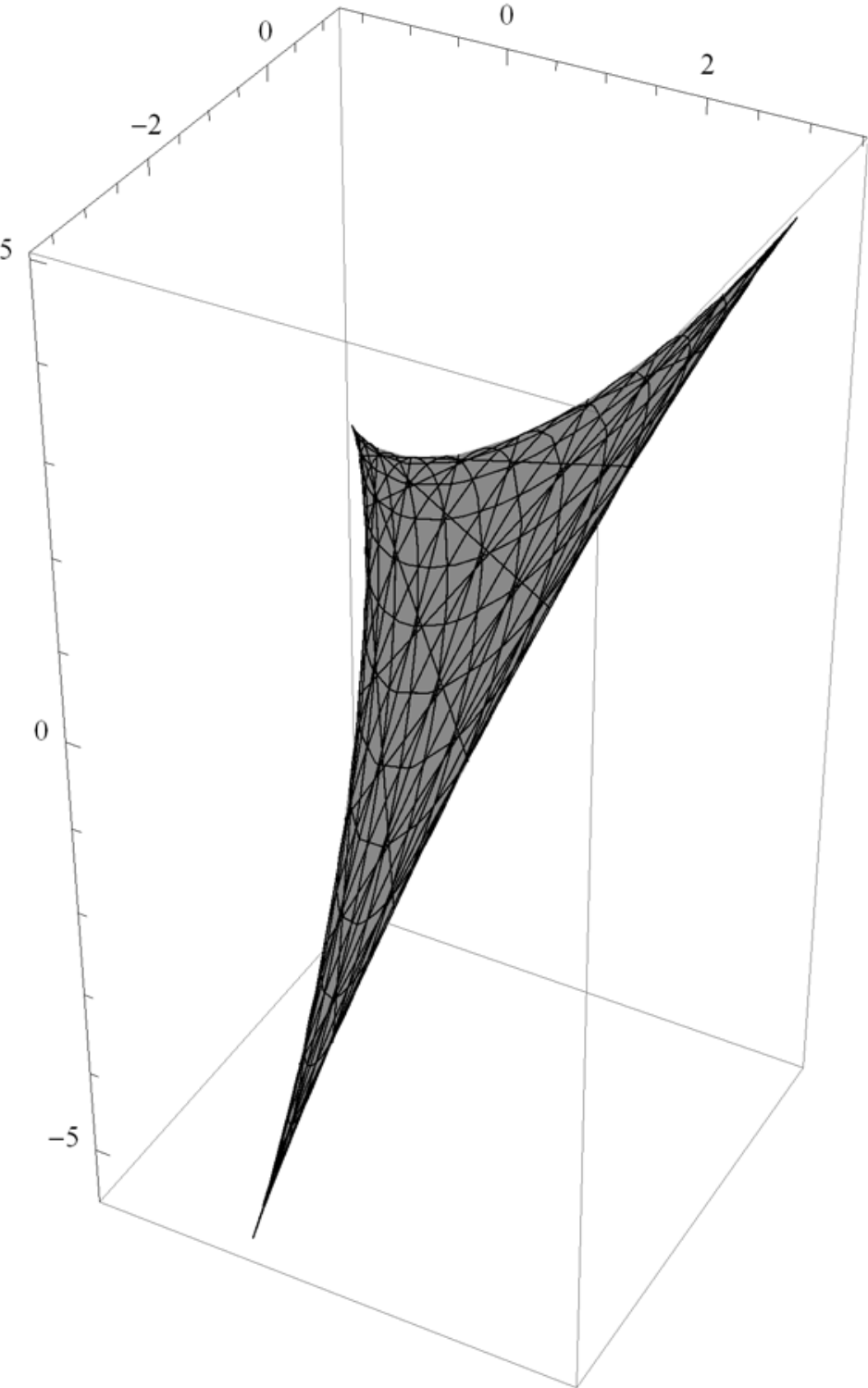}\\
\caption{A face  \(\varphi_1(H)\).}
\label{figure7.pdf}
\end{center}
\end{minipage}
\begin{minipage}{0.55\hsize}
\begin{center}
\includegraphics[scale=0.35]{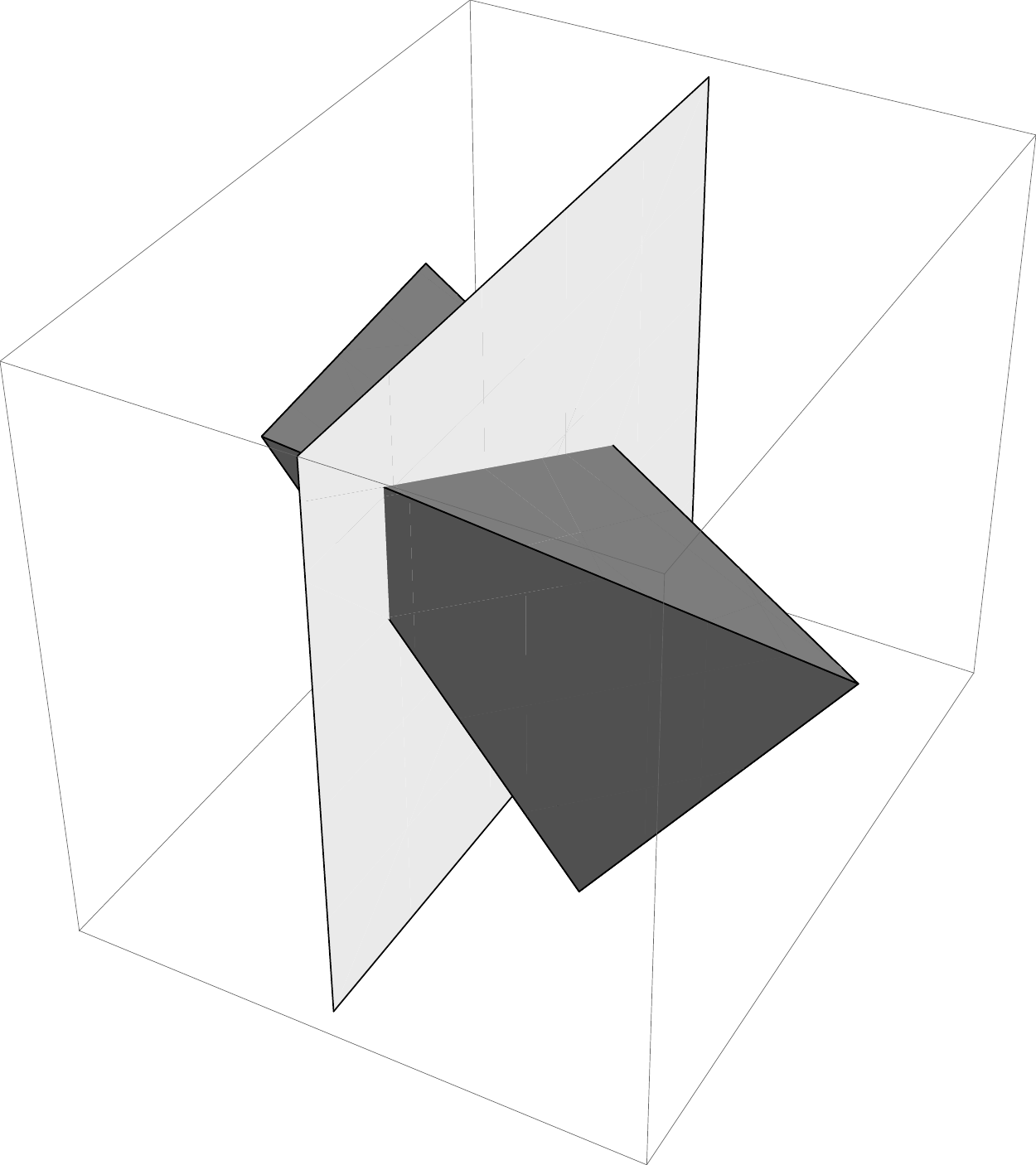}\\
\caption{A line segment $L$ and a face $H$.}
\label{figure8.pdf}
\end{center}
\end{minipage} 
\end{tabular}
\end{figure} 

\begin{pro} \label{pro:N2}
\(\varphi_1(L)\) is a ruling of \(\mathcal{A}\).  At any point of \(\varphi_1(L)\), the face \(\varphi_1(H)\) is tangent to \(\varphi_1(\triangle OB_2B_3)\).
\end{pro}
\begin{proof}
Let \enskip \((p_1, p_2, q) : = \varphi_1(c, \beta, -2c-\beta).\)  Then as in the proof of (2.8), we have 
\[(p_1, p_2, q)  = 2(\cos c,  \sin c, \cos 2c) + 2\cos(\beta+c)(\cos c, -\sin c, 2).\]  Hence from (2.8), we see that \(\varphi_1(L)\) is a ruling of \(\mathcal{A}\).
 
Since \enskip \(\triangle OB_2B_3 = \{(\alpha, \beta, \gamma) \in R' : 2\alpha + \beta + \gamma = 0\}\),\quad then \enskip \(\varphi_1(\triangle OB_2B_3)\) \enskip is given by 
\[p_1(\alpha, \beta) = 2\cos \alpha +  2\cos (\alpha+ \beta)\cos \alpha,\quad
p_2(\alpha, \beta) = 2\sin \alpha - 2\sin \alpha \cos(\alpha+ \beta),\]
\(q(\alpha, \beta) =  2( \cos 2\alpha + 2\cos (\alpha + \beta)).\)\\
Set\quad \(\chi(\alpha, \beta) = (p_1(\alpha, \beta),  p_2(\alpha, \beta),  q(\alpha, \beta)). \) \quad Let \enskip \(N : = (\cos c,  -\sin c, -1/2)\)  be the normal to  \(\varphi_1(H)\) at \(\varphi_1(c, \beta, -2c-\beta)\).
We see that  the normal vector \(N\) is also orthogonal to the tangent vectors
\[ \frac{\partial \chi}{\partial \alpha} \enskip \mbox{and }\enskip \frac{\partial \chi}{\partial \beta} \enskip \mbox{at }\enskip \varphi_1(c, \beta, -2c-\beta).\qquad \qquad \qquad \qquad\]
\end{proof}
We describe the 'inscribed face'  \(\varphi_1(H)\)  in Proposition 4.6 in the words of internal rays.
Set \enskip \(D_0(\beta) = \cup_{\alpha}R_0(\alpha, \beta, \alpha)\) .  Then  we have the  following proposition.
 \begin{pro} \label{pro:R3}
\(D_0(\beta)\)  is equal to \enskip \(\varphi_1(\{ \beta = constant \}) \) .  
\end{pro} 
\begin{proof}
If we regard \enskip\(\alpha + \theta\)\enskip as \(\alpha'\)   and   \(\alpha - \theta\) \enskip as \(\gamma'\) \enskip  in (4.9),  then we have \(z_1 = e^{i\alpha'} + e^{i\gamma'} +  e^{i\beta} + e^{-i(\alpha'+\beta+\gamma')}.\) \quad We fix \( \beta = constant \)   and move \(\alpha\)  and \(\theta\) .   Then we have   \(\varphi_1(\{ \beta = constant\}) = D_0(\beta)\) .
\end{proof}

Using external rays in \(R_3\) whose internal rays are of type (2)  in Proposition 4.4,  we construct a map \(E\) from  \(\mathcal{M}_0\)  to  \(\mathcal{A}_0\),
\[\mbox{where} \quad \mathcal{M}_0  = \{(e^{\theta i},\enskip xe^{\frac \theta 2 i }) :  0 \leq \theta < 2\pi, \enskip -2 < x < 2\}, \qquad\]
\[\mbox{and} \quad \mathcal{A}_0 = \{(4\cos^3 u , \enskip 4\sin^3 u,  \enskip 6\cos 2u) + v(\cos u , \enskip -\sin u ,\enskip 2 ) : 0 \leq u < 2\pi , \]
\[ -2-2\cos 2u < v  <  2-2\cos 2u\} .\qquad \qquad \qquad \qquad\]
The external ray  \(R(\alpha, \beta, \alpha)\) \enskip  with   \enskip   \(\alpha + \beta \ne 0,  \pi\) \enskip  has two end points.  One is in   \(\mathcal{M}_0\)  and the other is in   \(\mathcal{A}_0\).  Using these two end points, we define a map \(E\) from  \(\mathcal{M}_0\)  to  \(\mathcal{A}_0\) by 
\begin{equation}
\begin{split}
E(( e^{2i\alpha} : 2\cos(\alpha+\beta)e^{i\alpha} : 1 : 0))\qquad \qquad \qquad\\
= (2e^{i\alpha}  + e^{i\beta} + e^{i(-2\alpha-\beta)}, \quad 4\cos(\alpha+\beta) + 2\cos 2\alpha).
\end{split}
\end{equation}
 \begin{pro} \label{pro:R4}
The image of any ruling of \(\mathcal{M}_0\)  under the map \(E\) is a also a ruling of  \(\mathcal{A}_0\) . 
\end{pro} 
\begin{proof}
In (4.10), we fix  \(\alpha\)   and  move \(\beta\).  Then by the same argument used in the proof of Proposition 2.4, we can prove that the image \((2e^{i\alpha}  + e^{i\beta} + e^{i(-2\alpha-\beta)}, \quad 4\cos(\alpha+\beta) + 2\cos 2\alpha)\) is written as (2.8).
\end{proof}
\section{The set of critical values and Catastrophe theory}
In this section we show some relations between \(P_{A_3}^d\)  and catastrophe theory. Before we start studying the relations,  we review some result on maps  \( P^d_{A_2} \) on \( {\mathbb C}^2 \) related to the Lie algebra of type \(A_2\).  We show in \cite{U} the following results.   The set of critical values of \( P^d_{A_2} \) restricted to 
 \(\{z_1 = {\bar z}_2\}\) is a deltoid.  The deltoid coincides with a cross-section of the bifurcation set (caustics) of the elliptic umbilic catastrophe map  \((D_4^-)\).  The external rays and their extensions  constitute a family  of lines whose envelope is the deltoid.  These lines are real 'rays' of caustics.   See Figure \ref{figure9.pdf}.
\begin{figure}
\includegraphics[scale=0.4]{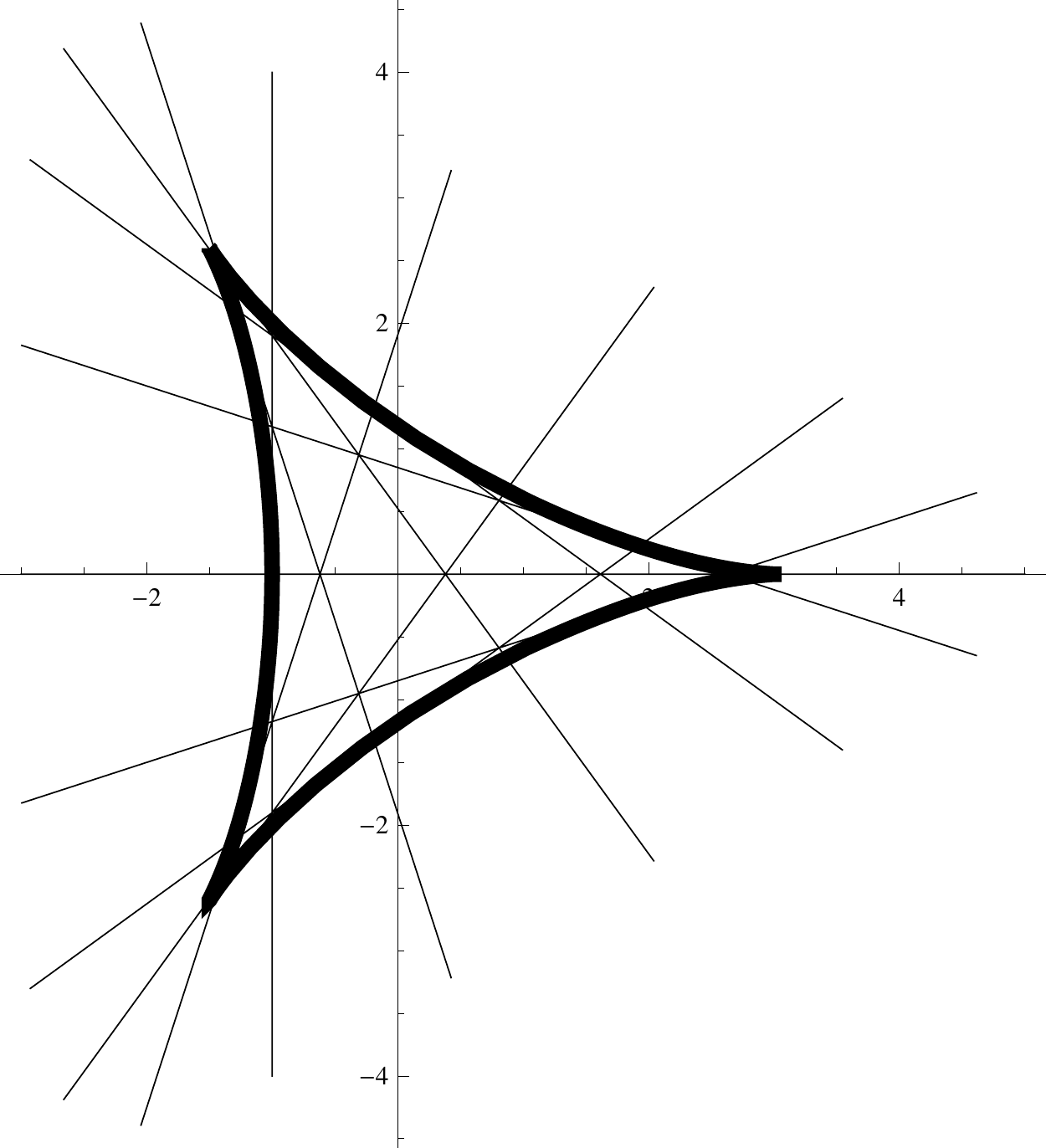}\hspace{1.5cm} \\
\caption{ A deltoid and external rays.}
\label{figure9.pdf}
\end{figure}
In addition to the caustics, the deltoid has relations with binary cubic forms
\[f(x,y) = ax^3 + bx^2y + cxy^2 + dy^3, \quad a,b,c,d \in {\mathbb R}.\]
The discriminant \(D\)  is given by 
\[D = 4(ac^3 + b^3d) +27a^2d^2 - b^2c^2 -18abcd .\]
\[\mbox{Set} \quad V =  \{( a, b, c, d) \in {\mathbb R}^4 : D( a, b, c, d) = 0\}.\]
Zeeman\cite{Z} shows that \(V \cap S^3\)  is mapped diffeomorphically to the 'umbilic bracelet'.  It has a deltoid section that rotates \(1/3\) twist going once round the bracelet.

Now we return to the study of the maps \(P_{A_3}^d\).
We will show that the set of critical values of \(P_{A_3}^d\) restricted to \(R_3\) decompose into the tangent developable of an astroid and two real curves. The set coincides with a cross-section of the 
set obtained by Poston and Stewart[9, 10] where binary quartic forms are degenerate.  The shape for the cross-section is called the 'Holy Grail'.

We begin with the study of the critical set of \(P_{A_3}^d\).  Let  \(t_4 = 1/(t_1t_2t_3)\).  We use the notation in (1.1).
 \begin{pro} \label{pro:T}   Critical set \(C_d\) of  \(P^d_{A_3}(z_1, z_2, z_3) \) is equal to

\[\{(z_1, z_2, z_3) \in  {\mathbb C^3} : t_1 = \varepsilon t_2 \enskip \mbox{or} \enskip t_1 = \varepsilon t_3 \enskip \mbox{or} \enskip t_1 = \varepsilon t_4 \enskip \mbox{or} \enskip \]
 \[ t_2 = \varepsilon t_3 \enskip \mbox{or} \enskip  t_2 = \varepsilon t_4\enskip \mbox{or} \enskip t_3 = \varepsilon t_4, \]
\[ \varepsilon = e^{2j\pi\sqrt{-1}/d} \quad (1 \le j \le d-1)\}.\]
\end{pro}
\begin{proof}  Recall the map  \( \Phi _1( t_1, t_2 ,t_3) =(z_1,z_2,z_3)\)
 .  Then 
\[\det D\Phi_1 = t_4\prod_{1 \le i < j \le 4}(t_i - t_j).\]
And
\[\det D(P_{A_3}^d\circ \Phi_1) = d^3t_4\prod_{1 \le i < j \le 4}(t_i^d - t_j^d).\]
The proposition follows because
\[\det DP_{A_3}^d = \det D(P_{A_3}^d\circ \Phi_1)/\det D \Phi_1 .\]
\end{proof}
  Clearly, the sets \(P_{A_3}^d(C_d)\)   \((d = 2, 3, 4 \cdot \cdot \cdot)\)   are the same.   The set  \(P_{A_3}^d(C_d)\) is an algebraic surface  in \( {\mathbb P^3}\) invariant  under \(P_{A_3}^d\), \\ i. e.,
\[P_{A_3}^d(P_{A_3}^d(C_d)) = P_{A_3}^d(C_d).\]
\(P_{A_3}^d \)  is a critically finite map.  See \cite{DS}.

We will determine the set \( P_{A_3}^d(C_d) \cap R_3\) .  We may set \(f : = P^2_{A_3}(z_1, z_2, z_3) \) \enskip and \enskip \(C : = C_2.\)   If \((z_1, z_2, z_3) \in C,\)  then without loss of generality we may assume that \(t_1 = -t_4,\) where \(t_4 =  1/(t_1t_2t_3).\)  Then 
\[z_1 =  t_2 + t_3 , \quad z_2 = t_2 t_3 +  \frac 1{t_2 t_3}, \quad z_3 =  \frac 1{t_2} + \frac 1{t_3} ,\]
and the image of \((z_1,z_2,z_3)\)  under \(f\) is written as 
\[z_1^{(2)} =  t_2^2 + t_3^2 - 2 \frac {1}{t_2t_3}, \qquad\]
\[z_2^{(2)} = t_2^2t_3^2 - 2(\frac {t_2}{t_3} + \frac {t_3}{t_2} ) + \frac 1{t_2^2 t_3^2} ,\]
\[z_3^{(d)}= \frac 1{t_2^2} + \frac1{t_3^2} - 2 t_2 t_3.\qquad \qquad \]
Set \(t_2 = re^{i\alpha}\) and \(t_3 = Re^{i\beta}\).  Then to determine the set
\(f(C) \cap R_3\)  we need the following. 
\begin{pro} \label{pro:V} The point \((z_1^{(2)}, z_2^{(1)}, z_3^{(2)})\)
belongs to the set \(R_3\) if and only if the following three conditions are satisfied :
\end{pro}
\begin{itemize}
\item [(1)] $(r^2R^4 - r^2)\cos {2b} + 2(r^3R^3 - rR)\cos {(a+b)} = R^2 - r^4R^2,$
\item  [(2)]$(r^2R^4 - r^2)\sin {2b} + 2(r^3R^3 - rR)\sin {(a+b)} = 0,$
\item  [(3)]$ (r^4R^4 - 1)\sin {a} - 2(r^3R - rR^3)\sin {b} = 0, \\
\mbox{ where} \quad a = 2\alpha+2\beta, \enskip b = \alpha - \beta.$
\end{itemize}

\begin{proof} \quad We may check the conditions 
\[z_1^{(2)} = \overline{z_3^{(2)}} \quad \mbox{and} \quad z_2^{(2)} \in  {\bf \mathbb R}.\]
The former condition is equivalent to
\[(r^2 - \frac 1{r^2}) + (R^2 - \frac 1{R^2})e^{2(\alpha - \beta)i} + 2 (rR - \frac 1{rR})e^{(3\alpha +  \beta)i} = 0.\]
The latter condition is equivalent to 
\[ r^2R^2 e^{2(\alpha + \beta)i} + \frac 1{r^2R^2} e^{-2(\alpha + \beta)i} -2(\frac rR e^{i(\alpha-\beta)} + \frac Rr e^{i(\beta-\alpha)}) \in  {\bf \mathbb R}.\]
Then the proposition follows.  
\end{proof}

Next we will show a refinement of  Proposition \ref{pro:V}.  We consider four cases :
\begin{enumerate}
\renewcommand{\labelenumi}{\roman{enumi})}
\item $r = R = 1,$ 
\item $rR = 1\quad \mbox{and} \quad r \ne R, $
\item $rR \ne 1 \quad \mbox{and} \quad r = R,$
\item $rR \ne 1 \quad \mbox{and} \quad r \ne R.$
\end{enumerate}
If \(r = R = 1\), then the conditions (1), (2) and (3) are trivially satisfied.
\begin{lemma} \label{lemma:l51}  We assume that the conditions (1), (2) and (3) in Proposition \ref{pro:V} are satisfied.
\begin{enumerate}
\renewcommand{\labelenumi}{\roman{enumi})}
\item $\mbox{If} \enskip rR = 1  \quad \mbox{and} \quad r \ne R, \quad  \mbox{then}\quad b =0,  \pi.$
\item $\mbox{If} \enskip rR \ne 1\quad \mbox{and} \quad r = R,\quad  \mbox{then} \quad (a,b) = (0, \pi), (\pi,0).$
\end{enumerate}
\end{lemma}
The proof is straightforward.
\begin{lemma} \label{lemma:l52}
We assume that \(rR \ne 1\quad \mbox{and} \quad r \ne R\).  Then there are not any numbers  \(0 < r,  R\quad \mbox{and} \quad 0 \le a, \enskip b < 2\pi\)  satisfying (1), (2), (3) in  Proposition \ref{pro:V}.
\end{lemma}
{\it Proof.} \quad Suppose that there exist numbers  \(0 < r,  R\quad \mbox{and} \quad 0 \le a, \enskip b < 2\pi\)  satisfying (1), (2), (3).   From (3) we have
\begin{equation} 
 \sin a = c_1 \sin b, \quad \mbox{where} \quad c_1 : = \frac{2(r^3R-rR^3)}{r^4R^4-1}.
\end{equation}
We square the both sides of (1) and (2).  Then we add left-hand sides and add right-hand sides.  Hence if   \(R \ne 1\),  then
\begin{equation}
\begin{split}
\cos(a-b) = \frac 1{2pq}(R^4(1-r^4)^2-p^2-q^2) = : c_2,\\ \mbox{where} \quad p = r^2R^4 - r^2 \quad \mbox{and} \quad q = 2(r^3R^3 -rR).
\end{split}
\end{equation}
(We denote the right hand side of (5.2) by \(c_2\).)  Applying the addition theorem to \(\cos(a-b)\)  and using (5.1),  we obtain
\begin{equation}
 \sin^2b = \frac{1-c_2^2}{1+c_1^2-2c_1c_2}.
\end{equation}
From (2) and (5.1), it follows that 
\[\cos a\sin b = c_3\cos b \sin b, \quad \mbox{where} \quad c_3 = \frac{-r(1+R^4)}{R(1+r^2R^2)}.\]
{ Case 1 :} \quad \(\sin b \ne 0.\)  Then
\begin{equation}
\cos a = c_3\cos b.
\end{equation}
Substituting  \(\sin a\) in (5.1)  and \(\cos a\) in (5.4)  for those  in (1)  and then substituting \(\sin^2 b\) in (5.3) for the result,   we have
\[\frac{(r-R)(r+R)(-1+r^2R^2)^2}{1+r^2R^2} = 0.\]
A contradiction.\\
{ Case 2 :} \quad \(\sin b = 0.\)  Then \(\sin a = 0\).
\[\mbox{If} \quad (a, b) = (0, 0) \quad \mbox{or} \quad (\pi, \pi), \quad \mbox{then} \quad (r+R)^2(r^2R^2-1) = 0.\]
\[\mbox{If} \quad (a, b) = (0, \pi) \quad \mbox{or} \quad (\pi, 0), \quad \mbox{then} \quad (r-R)^2(r^2R^2-1) = 0.\]
In any case, we have a contradiction.  

If  \(R = 1\),  we also have a contadiction. \qquad \(\Box\)\\

From  Lemma \ref{lemma:l52},  we know that  \(f(C) \cap R_3\)  decomposes into three cases:
\begin{enumerate}
\renewcommand{\labelenumi}{\roman{enumi})}
\item $r = R = 1,$ 
\item $rR = 1\quad \mbox{and} \quad r \ne R, $
\item $rR \ne 1 \quad \mbox{and} \quad r = R.$
\end{enumerate}
The first case : \(r = R = 1\).\\
The set \(\{(z_1^{(2)}, \quad z_2^{(2)}, \quad z_3^{(2)}) : r = R = 1\}\) is equal to the astroidalhedron \(\mathcal A\). 
This is a central part of the tangent developable in Figure 10. \\
\\
The second case : \(rR = 1\)  and  \(r \ne R\).   From Lemma \ref{lemma:l51}, it follows that  \(b = 0\) or \(\pi\).\\

If \(b = \pi\),\quad then \quad  \(\alpha - \beta = \pi \) and so \quad \(t_2 = re^{i \alpha }, \) \quad \(t_3 = -\frac 1{r}e^{i\alpha } \) . \\ 
Set \quad \(\theta  = -2\alpha . \)  Then we have a top bowl.  This is an upper part  of the tangent developable in Figure 10. \\
{\bf top bowl;}
\begin{equation}
\begin{split}
z_1^{(2)}=(r^2+\frac 1{r^2})e^{-i\theta} +2e^{i\theta},\enskip z_2^{(2)}=2(r^2+\frac 1{r^2}) + 2\cos{2\theta} ,\\
z_3^{(2)}=(r^2+\frac 1{r^2})e^{i\theta} +2e^{-i\theta}
\end{split}
\end{equation}\\
If \(b = 0\), \quad then \quad  \(\alpha - \beta = 0 \) and so \quad \(t_2 = re^{i \alpha }, \) \quad \(t_3 = \frac 1{r}e^{i\alpha }. \) \enskip Set \quad \(\theta  = -2\alpha . \)  Then we have a lower bowl.  This is a lower part  of the tangent developable in Figure 10.
 \\
{\bf lower bowl;}
\begin{equation}
\begin{split}
z_1^{(2)} = (r^2+\frac 1{r^2})e^{-i\theta} -2e^{i\theta},\enskip z_2^{(2)} = -2(r^2+\frac 1{r^2}) + 2\cos{2\theta} ,\\ 
z_3^{(2)} = (r^2 + \frac 1{r^2}) e^{i\theta} -2e^{-i\theta}.
\end{split}
\end{equation}\\
The third case : \(rR \ne 1\)  and  \(r = R\).  Then \enskip \((a,b) = (0, \pi)\)  or \((\pi, 0)\).\\
If \(a = 0\)  and  \(b = \pi\),\quad then \(t_2 = ir\), \quad \(t_3 = -ir\). Then we have top whiskers.  See Figure 10.\\
\\
{\bf top whiskers;}
\begin{equation}
z_1^{(2)} =-2(r^2+\frac 1{r^2}), \quad z_2^{(2)} =r^4 +\frac 1{r^4} + 4 , \quad
z_3^{(2)} = -2(r^2+\frac 1{r^2}).
\end{equation}
If \(a = \pi\)  and  \(b = 0\),\quad then \(t_2 = t_3 = re^{i\pi/4}\).  Then we have lower whiskers.  See Figure 10.\\
{\bf lower  whiskers;}
\begin{equation}
z_1^{(2)} =2i(r^2 +\frac 1{r^2}), \quad z_2^{(2)} =-r^4-\frac 1{r^4} - 4 , \quad
z_3^{(2)} = -2i(r^2 +\frac 1{r^2}).
\end{equation}
Hence 
\(f(C) \cap R_3\) \enskip decomposes into the astroidalhedron \(\mathcal{A}\), a top bowl, a lower bowl, top whiskers and lower whiskers.

Next we consider relations between\enskip \(f(C) \cap R_3 \)  \enskip and external rays.   
The half-lines (5.5) and (5.6)  with \enskip \(1 \le r \le \infty\) \enskip are external rays \enskip \(R(-\theta,  \theta, -\theta)\)  and \enskip \(R(-\theta,  \theta+ \pi, -\theta)\)  and land at points on the upper  and lower self-intersection lines, respectively.  By Propositions 2.4 and 4.4, we know that  adding an internal ray to the half-lines,  we have a tangent line to the astroid.\\
\begin{figure}
\includegraphics[scale=0.45]{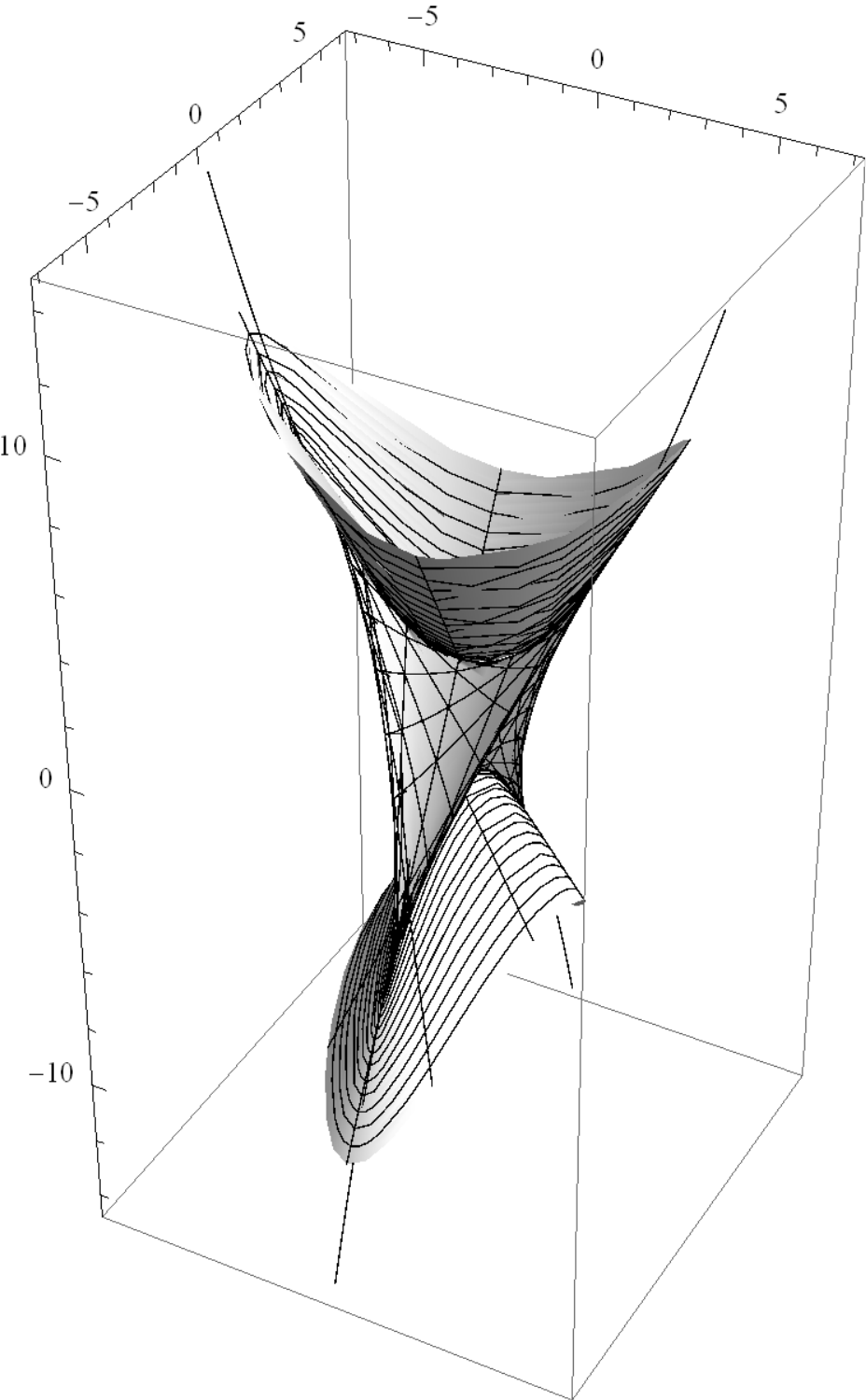}\hspace{1.5cm} \\
\caption{The tangent developable of an astroid in space and whiskers .}
\label{wiskers}
\end{figure}

Then we have the following proposition.
 \begin{pro} \label{pro:W}
\(f(C) \cap R_3 \setminus \{top \enskip and \enskip lower \enskip whiskers \}\) \enskip is the tangent  developable  \(\mathcal{T}\) of an astroid in space  given by 
\[\chi(u,v) = (4\cos^3 u , 4\sin^3 u,  6\cos 2u) + v(\cos u ,  -\sin u , 2 ) ,\qquad \qquad \]
\[ (-\infty < v  <  \infty ).\qquad \qquad \qquad \qquad\]
\end{pro} 
The tangent developable \(\mathcal{T}\) consists of  \(\mathcal{A}\), the top bowl and the lower bowl.  Any ruling of \(\mathcal{T}\) i.e. any tangent line to the astroid consists of two external  rays  and an intermediate internal ray.
 \begin{pro} \label{pro:X}
(1)  The rims of the bowls  join to the boundary of the M\(\ddot{o}\)bius strip  \(\mathcal{M}\)  in \(\Pi\).

(2) The images of the two self-intersection lines under the map \(\varphi\)  from  \(K(f)\) to  \(R\) defined in Section 2  are  two edges of the longest length  of the 
 \enskip  \((\sqrt 3, \sqrt {3}, 2)\)-tetrahedron \(\partial R\). 
\end{pro}
\begin{proof}\quad (1):  
The external rays in the top bowl and the lower bowl are given in (5.5) and (5.6).  Making  \enskip \(r \to \infty\) \enskip we see that 
\[\mbox{top bowl } : (z_1^{(2)} : z_2^{(2)} : z_3^{(2)} : 1) \to (e^{-i\theta} : 2 :  e^{i\theta} : 0)  \in \mathcal{M},\]
\[\mbox{lower  bowl } : (z_1^{(2)} : z_2^{(2)} : z_3^{(2)} : 1) \to (e^{-i\theta} : -2 :  e^{i\theta} : 0)  \in \mathcal{M}.\]

(2):  We denote four vertices of the \enskip \((\sqrt 3, \sqrt {3}, 2)\)-tetrahedron \(\partial R\)  by \enskip \(O = (0, 0, 0), \)\\
 \(A_1 = (0, -\pi/\sqrt 2, \pi), \quad A_2 = (\pi, 0, \pi) \) and  \(A_3 = (0, \pi/\sqrt 2, \pi).\)        See Figure 2.   The lengths of \(OA_2\) and \(A_1A_3\) are equal to  \(\sqrt 2 \pi\)  and the lengths of other edges are equal to  \(\sqrt 3 \pi /\sqrt 2 \).  The images of \(OA_2\) and  \(A_1A_3\) under the map  \(\varphi^{-1}\) are the upper self-intersection line and the lower self-intersection line, respectively.  See Figure 4. 
\end{proof}
Recall that \(J_3(f)\) is the closed domain bounded by \(\mathcal{A}\).   We have shown in Proposition 4.9 that the image of any ruling of  \(\mathcal{M}_0\)  under the map \(E\) is also a ruling of  \(\mathcal{A}_0\).  See Figures 11 and 12.
\begin{figure}[htbp]
\begin{tabular}{cc}
\begin{minipage}{0.52\hsize}\label{FigureCVA3}
\begin{center}
\includegraphics[scale=0.35]{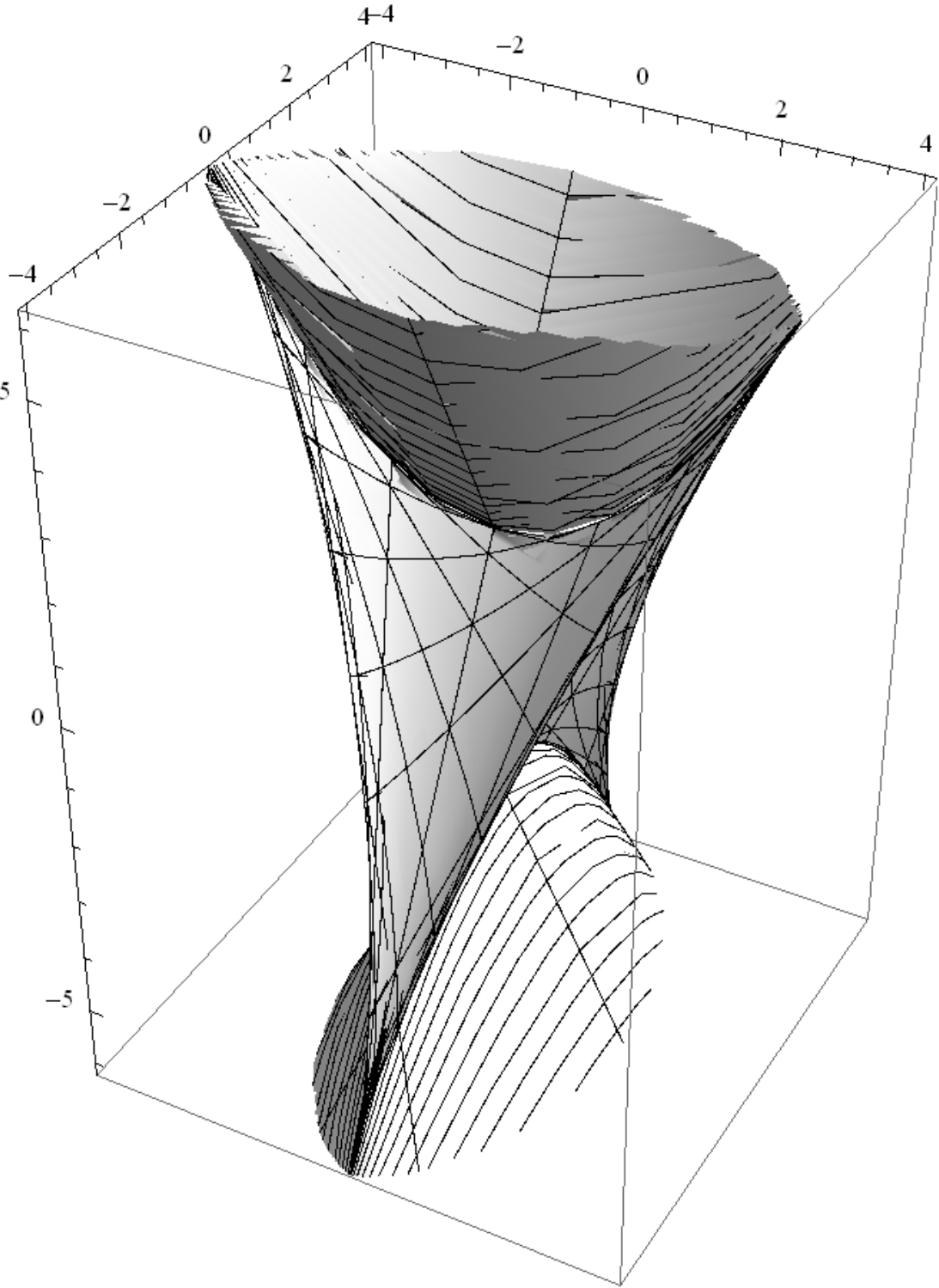} \\
\caption{ The tangent developable of an astroid in space.}
\end{center}
\end{minipage}
\begin{minipage}{0.48\hsize}\label{FigureMS}
\begin{center}
\includegraphics[scale=0.425]{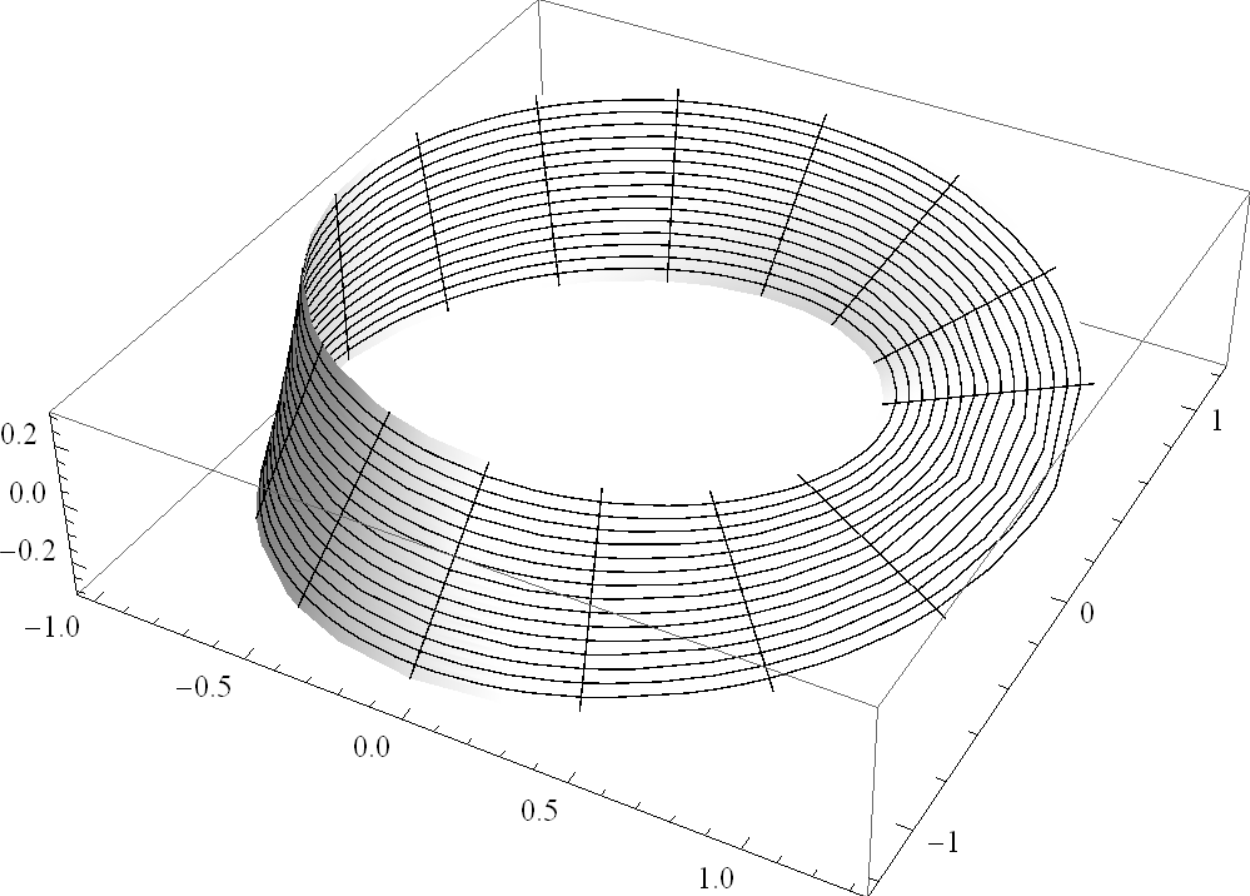} 
\caption{ A M\({\ddot o}\)bius strip.}
\end{center}
\end{minipage} 
\end{tabular}
\end{figure}

Lastly we consider relations between \enskip \(f(C) \cap R_3\) and binary quartic forms.
\quad Poston and Stewart  study  quartic forms in two variables in \cite{PS1} and \cite{PS2}
\[f(x,y) = ax^4 + 4bx^3y + 6cx^2y^2 + 4dxy^3 + ey^4, \quad a,b,c,d,e \in {\mathbb R}. \]
\(f(x,y)\)\enskip can be expressed uniquely as 
\begin{equation}
f(x,y) = Re(\alpha z^4 + \beta z^3\bar{ z} + \gamma z^2\bar{ z}^2), \quad \alpha, \beta \in {\mathbb C}, \quad \gamma \in {\mathbb R}.
\end{equation}
We use the results and notations in \cite{PS2}, pp.268-269.  
Let \(\triangle\) be the discriminant of \(f(x,y)\) and  \(\mathscr{Q} \subset {\mathbb R}^5\)     be the algebraic set given by \(\triangle = 0\).  To understand the geometry of \(\mathscr{Q}\)  they pursue a different tack.   The set \(\mathscr{W} = \mathscr{Q} \cap S^4\)  is decomposed  into \(\mathscr{W}_1\) and \(\mathscr{W}_{\infty}\).  \(\mathscr{W}_1\)  is diffeomorphic to  \(\mathscr{U}\). And  \(\mathscr{U}\) is the orbit of   \(\mathscr{Q}\)  under a maximal tours  \( {\mathbb T}\)  of \(GL_2( {\mathbb R})\).   \(\mathscr{Q}_0\) is the main part of   \(\mathscr{Q}\).  We consider the set \(\mathscr{Q}_0\). 
 Lemma 3.3 in \cite{PS1} states that  \(\mathscr{Q}_0\)  is given parametrically by 
\begin{equation}
 \beta = \frac 12(-3e^{i\phi} + e^{-3i\phi} -2\gamma e^{-i\phi}), \quad 0 \le \phi < 2\pi. 
\end{equation}
The shape for  \(\mathscr{Q}\)(or \(\mathscr{Q}_0\)) is called the Holy Grail in \cite{Ch}  and depicted in Fig. 5 in \cite{PS2}.   We compare the shape with Figure 11.  We show relations between  \(\mathscr{Q}_0\) and the tangent developable  \(\mathcal{T}\) in Proposition 5.5 of this paper. 
 \begin{lemma} 
The set \(\mathscr{Q}_0\) coincides  with \(\mathcal{T}\)  by a coordinate transformation.
\end{lemma} 
\begin{proof}
As in the proof of Lemma 3.3 in \cite{PS1},  we put \enskip \(\alpha = 1\)  and \enskip \(z = e^{i\theta}\) in the right-hand side of (5.9).  That is, we consider the equation 
\begin{equation}
 e^{4i\theta}  + e^{-4i\theta} + \beta e^{2i\theta} + {\bar{\beta}} e^{-2i\theta} +2 \gamma = 0.
\end{equation}
The equation  (5.10) is obtained by the considering condition that (5.11) has a double root in \(\theta\).  
We will find the same condition in our situation.   From (5.11), we have
\begin{equation}
 (e^{2i\theta})^4   + \beta (e^{2i\theta})^3 + 2\gamma (e^{2i\theta})^2 + {\bar{\beta}} e^{2i\theta} + 1 = 0.
\end{equation}
Hence we consider the equation
\begin{equation}
T^4 - z_1T^3 + z_2T^2 - z_3T +1 = 0.
\end {equation}
Let the solutions of (5.13) be \(t_1, t_2, t_3\) and \(t_4\).
Then the condition that (5.11) has a double root in \(\theta\) is described as follows.  From (5.12),  we assume that \enskip \(z_1 = {\bar z}_3\) \enskip and \(z_2\) is real.  That is, \enskip \((z_1, z_2, z_3) \in R_3\).  Under this assumption, (5.13) has a solution \(\{t_1, t_2, t_3, t_4\}\)  such that \(t_1 = t_2 = e^{i\theta}\).   Set \enskip \(t_3 = re^{i\phi}\).  Then \(t_4 = (1/r)e^{-i(2\theta + \phi)}\).   Relations between \(t_j\)'s and \(z_j\)'s are given in (1.1) with  \(t_4 = 1/(t_1 t_2 t_3)\).  Then we can express the condition  that such an element  \((z_1, z_2, z_3) \)  lies in \(R_3\) in the terms of the variables \(r, \phi\) and \(\theta\).
If \enskip \(r = 1\), then \enskip \((z_1, z_2, z_3) \in  \mathcal A\).  Next we assume that  \enskip \(r \ne 1\).  Then by an argument similar to that used in the proof of lemma \ref{lemma:l51} i), we see that if  such an element  \((z_1, z_2, z_3) \)  lies in \(R_3\) \enskip then \enskip \( \phi + \theta = 0\) \enskip or \enskip \( \phi + \theta = \pi\).    If \enskip \( \phi + \theta = 0\), \enskip then \((z_1, z_2, z_3)\) \enskip belongs to the top bowl in (5.5).  If \enskip \( \phi + \theta = \pi\), \enskip then \((z_1, z_2, z_3)\) \enskip belongs to the lower bowl in (5.6).   The coordinate transformation is given by \enskip \(\beta = -z_1\) \enskip and \enskip \(2\gamma = z_2\). 
\end{proof}
We can also prove this lemma by reparametrizing the ruled surface given by  (5.10) using a striction curve.

The set  \(\mathscr{Q} \setminus \mathscr{Q}_0\)  constitutes of two whiskers in \cite{PS2}.   
We can show that the whiskers in \cite{PS2}  coincide with the whiskers in (5.7) and (5.8) by the above coordinate transformation.  Each whisker in this paper joins to an attracting fixed point \enskip \(P_2 = (0 : 1 : 0 : 0)\)  of  \(f\).

\begin{pro} \label{pro:Z2} The set \(\mathscr{Q}\) coincides with \(f(C) \cap R_3\)  by a coordinate transformation. 
\end{pro}
In Proposition 5.6, we show that the rims of the bowls join to the boundary of \(\mathcal M\).  Poston and Stewart  deal with the same situation by considering the attaching map to   
 \(\mathscr{W}_{\infty} \subset S^2 = \{\alpha = 0\} \subset S^4\) in \cite{PS1} and \cite{PS2}.  But it is   complicated in \({\mathbb R}^5\).  But we consider the situation in \( {\mathbb P}^3( {\mathbb C})\).  Hence  the tangent developable  \(\mathcal T\) joins simply to the boundary of  \(\mathcal M\).   
We have  studied the external rays that connect \(\mathcal T\) and \(\mathcal M\) and any ruling of \(\mathcal T\) consists of two external rays and their intermediate interval ray. 

  We show the static aspect of catastrophe theory and also the dynamical aspect of catastrophe theory.

\bibliographystyle{amsplain}

\end{document}